\documentclass[a4paper, 11pt, oneside,reqno]{article}

\usepackage[utf8]{inputenc}
\usepackage[english]{babel}
\usepackage{caption}
\usepackage{float}
\usepackage{graphicx}
\usepackage{array}
\usepackage{hyperref} 
\usepackage[normalem]{ulem} 
\usepackage{enumitem} 

\usepackage{makecell} 

\usepackage[left=1in, right=1in, top=1in, bottom=1in, includefoot, headheight=13.6pt]{geometry}
\usepackage{url}
\usepackage{datetime}
\usepackage{times}
\usepackage[dvipsnames]{xcolor}
\usepackage{amsmath}
\usepackage{amssymb}
\usepackage{booktabs} 
\usepackage[titletoc,title]{appendix} 

\usepackage{tikz}

\usepackage{geometry} 
\usepackage{cases} 

\usepackage[amsmath,thmmarks]{ntheorem} 

\newtheorem{theorem}{Theorem}

\newtheorem{corollary}[theorem]{Corollary}
\newtheorem{lemma}[theorem]{Lemma}

\newtheorem{proposition}[theorem]{Proposition}

\theoremstyle{nonumberplain}
\theorembodyfont{\normalfont}
\theoremsymbol{\ensuremath{\square}}
\newtheorem{proof}{Proof}

\makeatletter
\newcommand{\leqnomode}{\tagsleft@true}
\newcommand{\reqnomode}{\tagsleft@false}
\makeatother

\newcommand\PBS[1]{\let\temp=\\%
  #1%
  \let\\=\temp
}

\newdateformat{dateenglish}{\monthname[\THEMONTH]~\THEDAY,~\THEYEAR}
\DeclareMathOperator{\sgn}{sgn}
\DeclareMathOperator{\Var}{Var} 
\DeclareMathOperator{\Cov}{Cov}
\DeclareMathOperator{\Cor}{Cor}

\let\P\relax 
\DeclareMathOperator{\P}{\mathbb{P}} 
\DeclareMathOperator{\E}{\mathbb{E}}

\def\R{\mathbb{R}}
\newcommand{\1}{\mbox{1\hspace{-0.28em}I}}
\newcommand\numberthis{\addtocounter{equation}{1}\tag{\theequation}}

\title{On the Dependence between Functions of \\Quantile and Dispersion Estimators} 

\author{\Large Marcel Br\"autigam$\dagger$ $\ddagger^{\;\ast}$ and Marie Kratz$\dagger^{\;\ast}$\\[1ex]
\small $\dagger$ ESSEC Business School Paris, CREAR \\ \small $\ddagger$ Sorbone University, LPSM \\ \small $^{\;\ast}$ LabEx MME-DII}

\date{}

\begin{document}

\maketitle

\begin{abstract}
\noindent
In this paper, we derive the joint asymptotic distributions of functions of quantile estimators (the non-parametric sample quantile and the parametric location-scale quantile estimator) with functions of measure of dispersion estimators (the sample variance, sample mean absolute deviation, sample median absolute deviation) - assuming an underlying identically and independently distributed sample. We also discuss the conditions required by the use of such estimators.
Further, we show that these results can be extended to any higher order absolute central sample moment as measure of dispersion. 
Aware of the difference in speed of convergence of the two quantile estimators, we compare the impact of the choice of the quantile estimator (and measure of dispersion) on the asymptotic correlations.
Then we prove a scaling law for the asymptotic dependence of quantile estimators with measure of dispersion estimators. 
Finally, we show a good finite sample performance of the asymptotics in simulations for elliptical distributions. 
All the results should constitute an important and useful complement in the statistical literature as those estimators are either of standard use in statistics and application fields, or should become as such because of weaker conditions in the asymptotic theorems.

\bigskip
\noindent {\emph 2010 AMS classification}: 60F05; 62H10; 62H20
\end{abstract}
%


\newpage

\section{Introduction and Notation}
\label{sec:Intro}

The joint asymptotic distribution between a sample quantile and a measure of location estimator, for an identically and independently distributed (iid) sample, has been considered in the literature for two location estimators, the sample median and the sample mean, respectively. Note that the case of the sample median (sample quantile itself) is included in the well-known asymptotics of a vector of sample quantiles.
The sample mean case was treated by \cite{Lin80} and later, using another approach, by \cite{Ferguson99}. These latter results have then been used by \cite{Bera16} to introduce a new characterization and hence also test for the normal distribution.
In this paper, we want to move from measures of location to measures of dispersion and present joint asymptotics for their functions with functions of quantile estimators, offering then a useful complement in the statistical literature.

As quantile estimators we consider, apart from the (non-parametric) sample quantile, also the parametric location-scale quantile estimator. The interest in considering two quantile estimators lies in the fact that, although being both consistent estimators, they have different speeds of convergence. Consequently, this impacts the joint asymptotic dependence with the measure of dispersion estimators. By measures of dispersion we mean well-known quantities as the variance or standard deviation, but also less frequently used ones as, for instance, the mean absolute deviation around the mean (MAD) or median absolute deviation around the median (MedianAD). The latter two have the advantage of relaxing the asymptotic constraints coming with the use of the sample variance (such as the existence of the fourth moment of the underlying distribution). For a more general and historical overview of measures of dispersion, we refer e.g. to \cite{David98}. 

Joint asymptotics between quantile estimators and measure of dispersion estimators have not yet been considered in full generality in the literature. Nevertheless a few examples exist. For instance, for symmetric location-scale distributions, the MedianAD equals the interquartile range (IQR), see e.g~\cite{Hampel74}, and their sample estimators are asymptotically equivalent, as shown in~\cite{Hall85}. It can be seen as a first contribution on joint asymptotics of quantile estimators and measure of dispersion estimators. Further, under some symmetry-type conditions on the underlying distribution, Falk proved  in \cite{Falk97} the asymptotic independence of the sample median (sample quantile of order 0.5) and the sample MedianAD.
In \cite{DasGupta06}, the joint asymptotics between the interquartile range and the standard deviation are shown.
An extension to higher moments has been given in \cite{Bos13}, where the authors provide the asymptotic joint distribution of the sample quantile with the r-th absolute sample moment in the case of a Gaussian distribution.
These results of \cite{Hall85}, \cite{Falk97}, \cite{DasGupta06}, \cite{Bos13} can be seen as special cases of our treatment of joint asymptotic distributions of functions of quantile estimators and functions of measure of dispersion estimators. 
%
In addition, while the results presented only hold for quantiles of order $p \in (0,1)$, one can see numerically that for $p$ tending to 0 or 1, we recover the asymptotic independence of the sample maximum/minimum and the measures of dispersion. The asymptotic independence of the sample maximum/minimum with the sample variance goes back to \cite{Loynes90} (and with the sample mean to \cite{Chow78}). 
For the MAD and the MedianAD no such results exist, to the best of our knowledge (for symmetric distributions, the sample MedianAD is asymptotically equivalent to half the sample IQR, thus the asymptotic independence in this case can be explained by the corresponding results on order statistics, see e.g. \cite{Falk88}).

Those asymptotic results should be of great use for applications in statistics or other application fields.
Note that the theoretical question arised from previous studies in financial risk management, one (see \cite{Brautigam18_WP}), where the correlation between a log-ratio of sample quantiles with the sample standard deviation is measured using log-returns from different stock indices, the other (see \cite{Zumbach18} and \cite{Zumbach12}) considering the correlation of ‘the realized volatilities with the centred volatility increment’ for different underlying processes.
A further application of the results presented in this paper may be  found in risk measure estimation: The sample quantile can be seen as a Value-at-Risk estimator and the functional framework allows us to extend the results to Expected Shortfall. 

The structure of the paper is as follows. We present in Section~\ref{sec:asympt_results} the main results about the asymptotic joint distribution and dependence between functions of quantile estimators (the sample quantile and the parametric location scale quantile) and functions of three different measures of dispersion (sample variance, sample MAD and sample MedianAD). 
Further, we analyse the effect of the sample size, when the quantile estimator and the measure of dispersion estimator have different sample sizes (in an asymptotic sense). 
We conclude the section discussing the conditions on the underlying distribution needed in the theorems and outline the methods used for proving the results. 
In Section~\ref{sec:examples}, we focus on two different applications.
First, we analyse the difference in the asymptotics depending on the choice of the quantile estimator and the measure of dispersion. 
Therefore, we compare the asymptotic correlations with the sample quantile versus the parametric location-scale quantile estimator (for each corresponding measure of dispersion). We discuss this in the cases of the two main elliptical distributions, the Gaussian and Student ones.
Second, we evaluate the finite sample approximation of the theoretical asymptotics: In a simulation study, we compare the sample correlation between quantile estimators and measure of dispersion estimators (each on a finite sample) to the theoretical asymptotic correlation - considering elliptical distributions with light and heavy tails, respectively.
We conclude in Section~\ref{sec:conclusio}.
Extensions of the main theorem and proofs are deferred to the Appendix.

{\sf \bf \large Notation}

Let $(X_1,\cdots,X_n)$ be a sample of size $n$, with parent random variable (rv) $X$, parent cumulative distribution function (cdf) $F_X$, (and, given they exist,) probability density function (pdf) $f_X$, mean $\mu$, variance $\sigma^2$, and quantile of order $p$ defined as $q_X(p):= \inf \{ x \in \R: F_X(x) \geq p \}$.
We denote its ordered sample by $X_{(1)}\leq ...\leq X_{(n)}$.
Whenever it exists, we introduce the standardised version (with mean 0 and variance 1) of $X$, namely $Y := \frac{X-\mu}{\sigma}$, and correspondingly the cdf, pdf and quantile of order $p$ as $F_Y$, $f_Y$ and $q_Y (p)$. 
In the special case of the standard normal distribution $\mathcal{N}(0,1)$, we use the standard notation $\Phi,\phi, \Phi^{-1} (p)$ for the cdf, pdf and quantile of order $p$, respectively.
We use the symbol $\sim$ for `distributed as', e.g. $Y \sim \mathcal{N}(0,1)$ means that $Y$ is $\mathcal{N}(0,1)$-distributed.

In this paper, we focus on the following five estimators.
First, we consider three estimators of the dispersion: \eqref{eq:D_in_sigma_n_sq}~the sample variance $\hat{\sigma}_n^2$, \eqref{eq:D_in_theta_n}~the sample mean absolute deviation around the sample mean (MAD) $\hat{\theta}_n$, and \eqref{eq:D_in_xi_n}~the sample median absolute deviation around the sample median (MedianAD) $ \hat{\xi}_n$. We introduce a unified notation for them:

\vspace*{-0.7cm}
\begin{minipage}{0.35\textwidth}
\begin{equation*}
D_{i}=  \begin{cases} \sigma^2 & \text{~for~} i=1,
\\ \theta & \text{~for~} i=2,
\\ \xi  & \text{~for~} i=3,
\end{cases}
 \text{~~and estimators~~~~~}
\end{equation*}
\end{minipage}
\begin{minipage}{0.65\textwidth}
\begin{numcases}{\hat{D}_{i,n}= } 
\hat{\sigma}_n^2:= \frac{1}{n-1} \sum_{j=1}^n (X_j - \bar{X}_n)^2, & for $i=1$,~ \label{eq:D_in_sigma_n_sq}
\\ \hat{\theta}_n := \frac{1}{n} \sum_{j=1}^n \lvert X_j - \bar{X}_n \rvert,  &for $i=2$,~ \label{eq:D_in_theta_n}
\\  \hat{\xi}_n:=  \frac{1}{2} (W_{(\lfloor \frac{n+1}{2}  \rfloor)} + W_{(\lfloor \frac{n+2}{2}  \rfloor)}), & for $i=3$,~ \label{eq:D_in_xi_n}
\end{numcases}
\end{minipage}

where $\bar{X}_n = \frac{1}{n} \sum_{j=1}^n X_j$, $W_j = \lvert X_j - \hat{\nu}_n \rvert, j=1,...,n$, and $\hat{\nu}_n = \frac{1}{2} (X_{(\lfloor \frac{n+1}{2}  \rfloor)} + X_{(\lfloor \frac{n+2}{2}  \rfloor)})$ (sample median of the original sample).
By $\lceil x \rceil =   \min{ \{ m \in \mathbb{Z}  : m \geq x \} }$, $\lfloor x \rfloor =   \max{ \{ m \in \mathbb{Z}  : m \leq x \} }$, we denote the rounded-up and rounded-off integer-parts of a real number $x \in \R$, respectively.

Note that we chose $\hat{\theta}_n$ with a factor of $\frac{1}{n}$ instead of $\frac{1}{n-1}$ to be in line with the literature (see e.g.~\cite{Gorard05},\cite{Pham01},\cite{Segers14}), and since it does not matter asymptotically.

Then, we consider two quantile estimators. 
Again, to have a unified notation, we introduce the quantile estimator $\hat{q}_n$ that may represent either the sample quantile $q_n$, or the location-scale quantile estimators $q_{n,\hat{\mu}, \hat{\sigma}}$ (or $q_{n, \hat{\sigma}}$ for $\mu$ known) whenever $F_X$ belongs to the location-scale family of distributions. Those estimators are defined as follows, for any order $p$,
\begin{align}
 q_n (p) &= X_{( \lceil np \rceil )},   \label{eq:qn}
\\ q_{n,\hat{\mu},\hat{\sigma}} (p) &= \hat{\mu}_n + \hat{\sigma}_n q_Y (p) \text{~~(and~} q_{n,\hat{\sigma}} (p) = \mu + \hat{\sigma}_n q_Y(p), \text{~respectively)}, \label{eq:q_mu_sigma}
\end{align}
where $\hat{\mu}_n$ and $\hat{\sigma}_n$ (by abuse of notation) are any estimators of the mean $\mu$ and standard deviation $\sigma$ . We choose them to be the sample mean $\bar{X}_n$ and the square root of the sample variance $\sqrt{\hat{\sigma}_n^2}$, respectively.

Recall that the location-scale family of distributions $\mathcal{F}$
is the class of distributions such that 
\begin{equation}\label{eq:loc-scale}
\text{if~} F \in \mathcal{F}, \,\text{then for any~} a\in \mathbb{R}, 0<b<\infty, ~G(x) := F(ax+b) \in \mathcal{F}.
\end{equation}
In addition,  to be consistent in the notation with related results in the literature, we generalise a notation used in \cite{Bera16} and \cite{Ferguson99}:
Assuming that the underlying rv $X$ has finite moments up to order $l$, and that $\eta$ is a continuous real-valued function, we set, for $1\leq k\leq l$ and $p \in (0,1)$,
\begin{align}
\tau_{k} (\eta(X),p) &= 
 (1-p) \left( \E[\eta^k (X) \vert X > q_X (p)] - \E[\eta^k(X)]\right) \label{eq:def-tau} 
\\ &= p (1-p) \left( \E[\eta^k (X) \vert X > q_X (p)] - \E[\eta^k(X) \vert X < q_X(p)]\right), \label{eq:def-tau2} 
\end{align} 
where the expression \eqref{eq:def-tau2} points out that this quantity involves the truncated moments of both tails.
When $\eta$ is the identity function, we abbreviate $\tau_k(X,p)$ as $\tau_k(p)$.

\vspace*{-.5cm}
Finally, the signum function is denoted by $\sgn$ and defined, as usual, by $\sgn(x) := \begin{cases}
-1 & \text{if } x < 0, \\
0 & \text{if } x = 0, \\
1 & \text{if } x > 0, \end{cases}$
and the standard notations $\overset{d}\rightarrow$ and $\overset{P}\rightarrow$ correspond to the convergence in distribution and in probability, respectively. Further, for a sequence of random variables $X_n$ and constants $a_n$, we denote by $X_n = o_P(a_n)$ the convergence in probability to 0 of $X_n/a_n$.
%

\section{Asymptotic Joint Properties of Quantile and Dispersion Estimators for iid rv's}
\label{sec:asympt_results}

\vspace{-1ex}
Let us present the main results, either with the sample quantile as quantile estimator (for any iid sample), or the location-scale quantile estimator (when considering location-scale distributions). Note that we refer to 'historical estimation' when estimating the quantile with the sample quantile,  as it is evaluated on the historical data sample.
In both cases, the asymptotic properties are subject to some smoothness and moment conditions.

We introduce on purpose two different ways of estimating the quantile
as it has some impact on their asymptotic covariance and correlation with the corresponding measure of dispersion (sample variance, sample MAD or sample MedianAD), thus in practice too.
Although both quantile estimators converge to the same quantity, the theoretical quantile, they do not have the same rate of convergence.
Using the location-scale quantile estimator, we obtain a better rate of convergence than with the historical estimation, as expected. Hence the interest of investigating this second way of estimation, to see how this impacts the dependence structure. 

Before stating the main results, let us present the different conditions the underlying random variable $X$ needs to fulfil. Depending on the choice of quantile estimator and measure of dispersion estimator, we impose three different types of conditions:
The existence of a finite $2k$-th moment for any integer $k>0$, 
the continuity or $l$-fold differentiability of the distribution function $F_X$ (at a given point or neighbourhood) for any integer $l> 0$, 
and the positivity of the density (at a given point or neighbourhood). Those conditions are named as:
\begin{align*}
&(M_K) &&\E[X^{2k}] < \infty, 
\\ &(C_0) && F_X \text{~is continuous}, 
\\ \phantom{text to make the distance less} &(C_l^{~'}) &&F_X \text{~is~} l\text{-times differentiable,} \phantom{text to make the distance between \& and \& \& less} 
\\ &(P) &&f_X \text{~is positive.} 
\end{align*}
Note that a standard condition often stated in the literature is $F_X$ to be absolutely continuous and strictly monotonically increasing. Clearly, this latter requirement is more general than our conditions $(C_0)$, $(P)$ almost everywhere. 

Finally, to have results as general as possible (in view of statistical applications), all along the paper we consider functions $h_1,h_2$ of the estimators that we assume to be continuous real-valued functions with existing derivatives denoted by $h_1'$ and $h_2'$ respectively. Note that in fact, to apply the Delta method, it suffices for the derivatives to exist only at the point where they are evaluated at. We will omit recalling it in the conditions of the results.

\subsection{Historical Estimation}
\label{ssec:hist-estim}

We present two results for the relation of functions of the sample quantile with functions of sample estimators of dispersion measures, one for the MAD and the variance, the other for the MedianAD. Both results provide the bivariate asymptotic normality of the estimators, but require a different mathematical framework to handle the dependence between the estimators. 
%
\begin{theorem} \label{th:Q-sigma}
Consider an iid sample with parent rv $X$ having existing (unknown) mean $\mu$ and variance $\sigma^2$. 
Assume conditions $(C_1^{~'}), (P)$ at $q_X(p)$ each, $(M_r)$ for $r=1,2$ respectively, as well as $(P)$ at $\mu$ for $r=1$. 
Then the joint behaviour of the functions $h_1$ of the sample quantile $q_n(p)$ (defined in \eqref{eq:qn}), for $p \in (0,1)$, and $h_2$ of the sample measure of dispersion $\hat{D}_{r,n}$ (defined in \eqref{eq:D_in_sigma_n_sq} or~\eqref{eq:D_in_theta_n}, for $r=1$ and $2$ respectively), is asymptotically normal:
\begin{equation}\label{eq:cdf-Q-sigma}
\sqrt{n} \, \begin{pmatrix} h_1(q_n (p)) - h_1(q_X(p)) \\ h_2(\hat{D}_{r,n})  - h_2(D_{r}) \end{pmatrix} \; \underset{n\to\infty}{\overset{d}{\longrightarrow}} \; \mathcal{N}(0, \Sigma^{(r)}), 
\end{equation}
where the asymptotic covariance matrix $\displaystyle \Sigma^{(r)}=(\Sigma^{(r)}_{ij}, 1\le i,j\le 2)$ satisfies 
\begin{align}
\Sigma^{(r)}_{11}&=\frac{p(1-p)}{f_X^2(q_X(p))} \,\left(h_1'(q_X(p))\right)^2 ; \quad  \Sigma^{(r)}_{22}=\left(h_2'(D_r)\right)^2 \, \Var\left(\lvert X - \mu \rvert^r + (2-r)(2F_X(\mu)-1)X \right) ;  \label{eq:cov-hist-general1}
\\  \Sigma^{(r)}_{12}&= \Sigma^{(r)}_{21} =
h_1'(q_X(p)) \,h_2'(D_r) \times \frac{\tau_r (\lvert X - \mu \rvert,p) +(2-r)(2F_X(\mu)-1) \tau_1 (p)}{f_X(q_X(p))},  \label{eq:cov-hist-general2}
\end{align}
$\tau_r$ being defined in \eqref{eq:def-tau}.

The asymptotic correlation between the functional $h_1$ of  the sample quantile and the functional $h_2$ of the measure of dispersion is - up to its sign $a_{\pm} = \sgn( \,h_1'(q_X(p)) \times h_2'(D_r))$ - the same whatever the choice of $h_1,h_2$:
\begin{equation} \label{eq:cor-hist-general}
\!\!\! \lim_{n \rightarrow \infty} \Cor\left(h_1(q_n(p)),h_2( \hat{D}_{r,n})\right) =  a_{\pm} \times \frac{\tau_r (\lvert X - \mu \rvert,p)  +(2-r)(2F_X(\mu)-1) \tau_1 (p)}{\sqrt{ p(1-p) \Var\left(\lvert X - \mu \rvert^r + (2-r) (2F_X(\mu)-1)X \right)}}.
\end{equation} 
\end{theorem}

Note that the choice of estimator for the measure of dispersion (sample variance, $r=2$, or sample MAD, $r=1$) has an impact on the required existence of moments of $X$. Indeed, for $r=2$, we require the existence of the fourth moment of $X$, while, for $r=1$, only a finite second moment.

Turning to the case with the sample MedianAD, the different dependence structure appears clearly in the expressions of the covariance and correlation when compared to Theorem~\ref{th:Q-sigma} (e.g. involving maxima - something we do not have in Theorem~\ref{th:Q-sigma}).
\begin{theorem} \label{thm:Q-MedianAD}
Consider an iid sample with parent rv $X$ with (unknown) median $\nu$, MedianAD $\xi$ and, if existing, mean $\mu$ and variance $\sigma^2$.
Assume conditions $(C_0)$ in neighbourhoods of $\nu \pm \xi$, $(C_1^{~'})$ at $q_X(p), \nu$ and $\nu \pm \xi$, and $(P)$ at $\nu, q_X(p)$ and at least at one of $\nu\pm \xi$ each. Then
the joint behaviour of the functions $h_1$ of the sample quantile $q_{n}(p)$ (for $p \in (0,1)$) and $h_2$ of the sample MedianAD $\hat{\xi}_n$ (defined in \eqref{eq:D_in_xi_n} or Table~\ref{tbl-notation}) is asymptotically normal:
\[ \sqrt{n} \, \begin{pmatrix} h_1(q_n (p)) - h_1(q_X(p)) \\ h_2(\hat{\xi}_n)  - h_2(\xi) \end{pmatrix} \; \underset{n\to\infty}{\overset{d}{\longrightarrow}} \; \mathcal{N}(0, \Gamma), \]
where the asymptotic covariance matrix $\displaystyle \Gamma=(\Gamma_{ij}, 1\le i,j\le 2)$ satisfies 
\begin{align} 
\Gamma_{11}&=\frac{p(1-p)}{f_X^2(q_X(p))} \,\left( h_1'(q_X(p))\right)^2 ; \qquad  
\Gamma_{22}=\frac{1+ \gamma / f_X^2(\nu)}{4(f_X(\nu + \xi) + f_X(\nu-\xi))^2}\,\left(h_2'(\xi)\right)^2;   \label{eq:cov-hist-general1-MedianAD}
\\  \Gamma_{12}&= \Gamma_{21} = \, h_1'(q_X(p)) \,h_2'(\xi) \times
\label{eq:cov-hist-general2-MedianAD}
\\ & 
\frac{- \max{\left(0, \,F_X(\nu +\xi) - \max{\left( F_X(\nu - \xi),p\right)}\right)} + \frac{1-p}{2} + \frac{f_X(\nu+ \xi) - f_X(\nu -\xi)}{f_X(\nu)} \max{\left(-\frac{p}{2},\frac{p-1}{2} \right)} }{ f_X(q_X(p))\, \left(f_X(\nu + \xi) + f_X(\nu-\xi)\right)} \nonumber
\end{align}
with $\gamma := \left(f_X(\nu+ \xi) - f_X(\nu -\xi)\right) f_X(\nu)\,\left(f_X(\nu+ \xi) - f_X(\nu -\xi) -4\right) \, \left(1-F_X(\nu - \xi) - F_X(\nu + \xi)\right).$

The asymptotic correlation between the two functions 
is - up to its sign $a_{\pm} := \sgn(\,h_1'(q_X(p)) \,h_2'(\xi))$ - the same whatever the choice of $h_1,h_2$: 
$\displaystyle \quad \lim_{n \rightarrow \infty} \Cor\left( h_1(q_n(p)), h_2(\hat{\xi}_n)\right) = a_{\pm} \times $
\begin{align} 
 & \frac{- \max{\left(0, F_X(\nu +\xi) - \max{\left(F_X(\nu - \xi),p\right)}\right)} + \frac{1-p}{2} + \frac{f_X(\nu+ \xi) - f_X(\nu -\xi)}{f_X(\nu)} \max{\left(-\frac{p}{2}, \frac{p-1}{2} \right)}}{\sqrt{\frac{p(1-p)}{4}} \,\sqrt{1 + \frac{\gamma}{f_X^2(\nu)} }} \, .  \label{eq:cor-functional-sample-quant-MedianAD}
\end{align} 
\end{theorem}
Two remarks can be made with respect to the result presented. First, the asymptotic dependence with the sample MedianAD does not even require a finite mean. Second, for symmetric distributions, it holds that $f_X(\nu+ \xi) = f_X(\nu -\xi)$ and $\gamma=0$, so the expressions of the asymptotic covariance matrix $\Gamma$ and \eqref{eq:cor-functional-sample-quant-MedianAD} simplify a lot: Then the asymptotic correlation is independent of the specific underlying distribution (see \cite{Brautigam19_iid}).

\subsection{Location-Scale Quantile} \label{ssec:loc-scale-estim}

As a comparison to using historical estimation via sample quantiles (denoted by $q_n$), let us estimate the quantile via the known analytical formula for the quantile of the model, considering a given location-scale distribution with unknown but finite mean $\mu$ and variance $\sigma^2$ as defined in \eqref{eq:loc-scale}.
Consequently, we can write the quantile of order $p$ in such cases as
\begin{equation}\label{eq:q-loc-scale}
q_X (p) = \mu + \sigma q_Y (p),
\end{equation}
where $Y$ is the corresponding rv with standardised distribution having mean 0 and variance 1.

Hence, if we estimate $\mu$ by the sample mean $\bar{X}_n$ and $\sigma$ by the square-root of the sample variance, $\sqrt{\hat{\sigma}_n^2}$, the estimator defined in \eqref{eq:q_mu_sigma} (based on \eqref{eq:q-loc-scale}) can be written as
\begin{equation*} \label{eq:q-location-scale-model}
q_{n, \hat{\mu}, \hat{\sigma}} (p) = \bar{X}_n + \hat{\sigma}_n q_Y(p).
\end{equation*}
%
We keep the same structure as in Subsection~\ref{ssec:hist-estim}: First, we present a unified result for the dependence of (here) the location-scale quantile estimator with the sample variance or sample MAD. Then, we present the corresponding result when using the sample MedianAD.
Note that in the case $\mu$ known, the estimator reduces to
$\displaystyle q_{n, \hat{\sigma}} (p) = \mu + \hat{\sigma}_n q_Y(p)$ (as given in \eqref{eq:q_mu_sigma}), and studying the dependence with the dispersion measure estimators becomes simpler; we refer to \cite{Brautigam19_iid} in such a case.

Let us start with presenting the analogon of Theorem~\ref{th:Q-sigma} for functions of the location-scale quantile estimator.
\begin{proposition}\label{prop-q_mu_sigma-general}
Consider an iid sample with parent rv $X$ having existing (unknown) mean $\mu$ and variance $\sigma^2$.
Assume conditions $(M_2)$, as well as $(C_0)$ at $\mu$ if $r=1$.
Then, taking $r=1,2,$ the joint behaviour of the functions $h_1$ of the quantile estimator $q_{n, \hat{\mu}, \hat{\sigma}}(p)$ from a location-scale model (for $p \in (0,1)$) and $h_2$ of the measure of dispersion $\hat{D}_{r,n}$ (defined in \eqref{eq:D_in_sigma_n_sq} or~\eqref{eq:D_in_theta_n}, for $r=1$ and $2$, respectively) is asymptotically normal:
\begin{equation}
\sqrt{n} \, \begin{pmatrix} h_1(q_{n, \hat{\mu}, \hat{\sigma}} (p)) - h_1(q_X(p)) \\ h_2(\hat{D}_{r,n})  - h_2(D_r) \end{pmatrix} \; \underset{n\to\infty}{\overset{d}{\longrightarrow}} \; \mathcal{N}(0, \Lambda^{(r)}), 
\end{equation}
where the asymptotic covariance matrix $\displaystyle \Lambda^{(r)}=(\Lambda^{(r)}_{ij}, 1\le i,j\le 2)$ satisfies 
\begin{align}
\Lambda^{(r)}_{11}&=\sigma^2 \,\left(h_1'(q_X(p))\right)^2 \Big( 1+q_Y(p) \Big( q_Y(p) (\E[Y^4]-1)/4+\E[Y^3] \Big) \Big) ; \quad   \phantom{xxxxxxxxxxxxxx}
\\ \Lambda^{(r)}_{22}&=\left(h_2'(D_r)\right)^2 \, \Var\Big(\lvert X - \mu \rvert^r + (2-r)(2F_X(\mu)-1)X \Big) ;  \label{eq:cov_qsigma_locationscale-general1}
\\  \Lambda^{(r)}_{12}&= \Lambda^{(r)}_{21} = \sigma^{r+1} \,h_1'(q_X(p)) \,h_2'(D_r) \times  \label{eq:cov_qsigma_locationscale-general2} 
\end{align} 
{\small
\[ \hspace*{-1cm}\footnotesize \left(\E[Y^{r+1}] + (2-r) \left( 2 F_Y(0) -1 - 2 \E[\lvert Y \rvert^{r+1} \1_{(Y<0)} ] \right)+ \frac{q_Y(p)}{2}  \Big( \E[\lvert Y \rvert^{r+2}] - \E[\lvert Y \rvert^{r}] + (2-r) (2 F_Y(0)-1) \E[Y^3]\Big) \right).   \]
}
The asymptotic correlation between the functional $h_1$ of  the location-scale quantile estimator and the functional $h_2$ of the measure of dispersion is - up to its sign $a_{\pm} = \sgn( \,h_1'(q_X(p)) \times h_2'(D_r))$ - the same whatever the choice of $h_1,h_2$:
{\small
\begin{align} 
&\lim_{n \rightarrow \infty} \Cor\left(h_1(q_{n, \hat{\mu}, \hat{\sigma}}(p)), h_2(\hat{D}_{r,n}) \right) = a_{\pm} \times 
\label{eq:cor_qsigma_locationscale-general} \\
&  \frac{\E[Y^{r+1}] + (2-r) \Big( 2 F_Y(0) -1 - 2 \E[\lvert Y \rvert^{r+1} \1_{(Y<0)}]  \Big)+ \frac{q_Y(p)}{2}  \Big( \E[\lvert Y \rvert^{r+2}] - \E[\lvert Y \rvert^{r}] + (2-r) (2 F_Y(0)-1) \E[Y^3]\Big) }{\sqrt{\left(1+q_Y(p) \Big(q_Y(p) \frac{\E[Y^4] -1}{4} +  \E[Y^3]\Big) \right)\, \Var\Big(\lvert Y \rvert^r + (2-r) (2F_Y(0)-1) Y\Big)}}. \notag
\end{align}
}
%
\end{proposition}
%
Note that using the location-scale quantile model implies assuming the existence of a finite fourth moment with the sample variance and with the sample MAD - this is in contrast to the historical estimation with the sample quantile.

We now consider the joint asymptotics of functions of the location-scale quantile estimator with functions of the sample MedianAD.
\begin{proposition}\label{prop-q_mu_sigma-MedianAD}
Consider an iid sample with parent rv $X$ from a location-scale distribution having (unknown) median $\nu$, MedianAD $\xi$, mean $\mu$ and variance $\sigma^2$.
Under $(M_2)$, $(C_0)$ in neighbourhoods of $\nu \pm \xi$, $(C_1^{~'})$ at $\nu, \nu \pm \xi$, $(P)$ at $\nu$, and at least at one of $\nu\pm \xi$ each, the joint behaviour of the functions $h_1$ of the quantile estimator $q_{n, \hat{\mu}, \hat{\sigma}}(p)$ from a location-scale model (for $p \in (0,1)$) and $h_2$ of the sample MedianAD $\hat{\xi}_n$ (defined in \eqref{eq:D_in_xi_n} or Table~\ref{tbl-notation}) is asymptotically normal:
\begin{equation}
\sqrt{n} \, \begin{pmatrix} h_1(q_{n, \hat{\mu},\hat{\sigma}} (p)) - h_1(q_X(p)) \\ h_2(\hat{\xi}_n)  - h_2(\xi) \end{pmatrix} \; \underset{n\to\infty}{\overset{d}{\longrightarrow}} \; \mathcal{N}(0, \Pi), 
\end{equation}
where the asymptotic covariance matrix $\displaystyle \Pi=(\Pi_{ij}, 1\le i,j\le 2)$ satisfies 
\begin{align}
\Pi_{11}&=\sigma^2 \,\left(h_1'(q_X(p))\right)^2 \,\Big(1+q_Y(p) \Big(q_Y(p) (\E[Y^4]-1)/4+\E[Y^3]\Big)\Big)  \,; \quad  
\\ \Pi_{22}&=\frac{1+ \gamma/f_X^2(\nu)}{4 (f_X(\nu +\xi) + f_X(\nu - \xi))^2} \,\left(h_2'(\xi)\right)^2;  
\\  \Pi_{12}&= \Pi_{21}  = \frac{\, h_1'(q_X(p))\, h_2'(\xi) \, \sigma^2}{2\left(f_Y(\frac{\nu+ \xi -\mu}{\sigma}) +f_Y(\frac{\nu- \xi -\mu}{\sigma})\right)} \times  \footnotesize  \Bigg( \left. - \E\left[(Y^2 q_Y (p)+ 2Y) \1_{\left(\frac{\nu - \xi -\mu}{\sigma} < Y \leq \frac{\nu + \xi -\mu}{\sigma}\right)}\right] \right.   \label{eq:cov_qsigma_locationscale-MedianAD}
\\ &  
\hspace{-.8cm} \left.  + \frac{f_Y(\frac{\nu+ \xi -\mu}{\sigma}) -f_Y(\frac{\nu- \xi -\mu}{\sigma})}{f_Y (\frac{\nu-\mu}{\sigma})} \E\left[(Y^2 q_Y (p)  +2Y) \1_{\left(Y \leq \frac{\nu - \mu}{\sigma}\right)}\right]+ \frac{q_Y(p)}{2}\left(1-\frac{f_Y(\frac{\nu+ \xi -\mu}{\sigma}) -f_Y(\frac{\nu- \xi -\mu}{\sigma})}{ f_Y(\frac{\nu-\mu}{\sigma})} \right)\right.  \Bigg).\notag 
\end{align}
The asymptotic correlation remains independent - up to its sign $a_{\pm} = \sgn(h_1'(q_X(p)) h_2'(\xi)) $ - of the specific choice of $h_1,h_2$:
{\small
\begin{align} 
& \lim_{n \rightarrow \infty} \Cor\left( h_1(q_{n, \hat{\mu}, \hat{\sigma}}(p)), h_2(\hat{\xi}_n) \right) =  a_{\pm} \times \label{eq:cor_qsigma_locationscale-MedianAD}
\\ & 
\footnotesize \hspace{-1.9cm} \frac{ - \E\left[(Y^2 q_Y (p)  +2Y) \1_{\left(\frac{\nu - \xi -\mu}{\sigma} < Y \leq \frac{\nu + \xi -\mu}{\sigma}\right)}\right] + \frac{f_Y(\frac{\nu+ \xi -\mu}{\sigma}) -f_Y(\frac{\nu- \xi -\mu}{\sigma})}{f_Y (\frac{\nu-\mu}{\sigma})}  \E\left[\left(Y^2 q_Y (p)  +2Y\right) \1_{\left(Y \leq \frac{\nu - \mu}{\sigma}\right)}\right] + \frac{q_Y (p)}{2}\!\!\left(1-\frac{f_Y(\frac{\nu+ \xi -\mu}{\sigma}) -f_Y(\frac{\nu- \xi -\mu}{\sigma})}{f_Y (\frac{\nu-\mu}{\sigma})} \right) } {\sqrt{ 1 + \frac{\left(f_Y(\frac{\nu+ \xi -\mu}{\sigma}) -f_Y(\frac{\nu- \xi -\mu}{\sigma})\right)^2}{f_Y^2 (\frac{\nu-\mu}{\sigma})}  - \frac{4 \left(f_Y(\frac{\nu+ \xi -\mu}{\sigma}) -f_Y(\frac{\nu- \xi -\mu}{\sigma})\right)}{f_Y (\frac{\nu-\mu}{\sigma})} \left(1 - F_Y(\frac{\nu- \xi-\mu}{\sigma}) - F_Y(\frac{\nu+\xi-\mu}{\sigma})\right)} \times\sqrt{1+ q^2_Y(p) \frac{\E[Y^4]-1}{4} + q_Y(p) \E[Y^3]}} . \notag
\end{align}
}
%
\end{proposition}
While in the case of the asymptotics of the sample MedianAD with the sample quantile, we did not even require a finite mean of the underlying distribution, here, with the location-scale quantile estimator, we need a finite fourth moment.

\subsection{The Effect of Sample Size on the Asymptotics} \label{ssec:sample_size_theor}

We analyse how the asymptotic dependence of the quantile estimator with the measures of dispersion estimator (sample variance, sample MAD, sample MedianAD) is influenced by the chosen (asymptotic) sample size $n$.

Instead of looking separately at all the different cases when using either the historical estimation of the quantile or the location scale model, and one of the three measures of dispersion, we consider a unified approach.
This is motivated by the fact that we are more interested in seeing how the values change with different sample sizes, than presenting formulae for each of the sub-cases, which, anyway, can be deduced from the previous results. Further, we will observe the same scaling property for each measure of dispersion; thus, a unified approach seems appropriate.

We consider sample sizes, say $n_v=vn$ and $n_w=wn$ for integers $v,w>0$, so that they are asymptotically multiples of each other, {\it i.e.} $\displaystyle \lim_{n \rightarrow \infty} n_v/n_w =v/w$. This is a way to introduce `different' sample sizes into an asymptotic framework.
Also, for the sake of readibility, we use here a notation for the quantile estimator, namely $\hat{q}_n$, which only implicitly involves the order of the quantile $p$. With those notations, we can prove the following result.
\begin{theorem}[Asymptotic Scaling Law] \label{thm:general-longer-sample}
Let $v,w$ be positive integers and consider an iid sample with parent rv $X$ with, if existing, mean $\mu$ and variance $\sigma^2$. Under appropriate moment and continuity conditions for $X$ (i.e. $(M_1),(M_2),(C_0), (C_1^{~'}),(P)$, depending on the estimators), the asymptotic covariance between functions of a quantile estimator with sample size $n_v$, 
$h_1(\hat{q}_{n_v})$, and functions of the measure of dispersion estimator for $i \in \{1,2,3 \}$ with sample size $n_w$, $h_2(\hat{D}_{i,n_w})$, satisfies
\begin{equation} \label{eq:cov_longer_samplesize-general}
 \lim_{n \rightarrow \infty} \Cov\left(\sqrt{n} \,h_1(\hat{q}_{n_v}), \sqrt{n} \,h_2(\hat{D}_{i,n_w})\right) = 
 \frac1{\max{\left(v,w\right)}}\, \displaystyle \lim_{n \rightarrow \infty} \Cov\left(\sqrt{n} \,h_1(\hat{q}_{n}), \sqrt{n} \,h_2(\hat{D}_{i,n})\right).
\end{equation} 
Accordingly, the asymptotic correlation is given by
\begin{equation} \label{eq:cor_longer_samplesize-general}
 \lim_{n \rightarrow \infty} \Cor\left( h_1(\hat{q}_{n_v}), h_2(\hat{D}_{i,n_w})\right) = 
 \sqrt{\frac{\min{\left(v,w\right)}}{\max{\left(v,w\right)}}}\, \displaystyle \lim_{n \rightarrow \infty} \Cor\left( h_1(\hat{q}_{n}),  h_2(\hat{D}_{i,n})\right).
\end{equation} 
\end{theorem}
\vspace*{-0.5cm}
Note that the conditions of applicability (i.e. moment and continuity conditions on rv $X$) of this proposition depend on the chosen estimators ($q_{n,t}$ or $q_{n, \hat{\mu}, \hat{\sigma}}, q_{n ,\hat{\sigma},t}$, as well as $\hat{D}_{i,n, t}$ for $i=1,2,3$) and are the same as in the corresponding cases in Sections~\ref{ssec:hist-estim} and~\ref{ssec:loc-scale-estim}.

\subsection{Discussion} \label{ssec:discussion}

\vspace{-1ex}
Let us address different points regarding the main results presented in Sections~\ref{ssec:hist-estim} and~\ref{ssec:loc-scale-estim}.
First, we collect some remarks and implications, which directly follow from the asymptotic theorems.
Then, we review the conditions imposed by them on the underlying distribution $F_X$. 
Finally, we briefly outline the methods used in the proofs of the theorems and propositions. \\[-6ex]
\paragraph{Remarks -} First, regarding the generality of the results presented:
\begin{itemize}
\vspace*{-0.5cm}
\item[-] Theorem~\ref{th:Q-sigma} can be seen as the asymptotics of the sample quantile with the r-th absolute sample moment $\frac{1}{n} \sum_{i=1}^n \lvert X_i - \bar{X}_n \rvert^r$ for $r=1,2$. With some extra work, this result can be extended to any positive integer $r$ and can be found in Appendix~\ref{sec:Appendix_extension}. Here, we prefered to focus on the most well-known measures of dispersion.\\[-5ex]
\item[-] All results could be extended to a more general form using a function $h(x,y) = \begin{pmatrix} h_1 (x,y) \\ h_2(x,y) \end{pmatrix} $ or considering a vector-valued version of the theorem (vector of sample quantiles). Again, this can be found in Appendix~\ref{sec:Appendix_extension} as we prefered, for readability, the above presentation.
\end{itemize}
\vspace*{-0.5cm}
Further, it is worth to point out the following properties when $X$ follows a location-scale distribution:
\vspace*{-.5cm}
\begin{itemize}
\item[-] For the class of location-scale distributions, all the asymptotic correlations do depend neither on the location parameter $\mu$, nor on the scale parameter $\sigma$.\\[-4ex]
\item[-] The asymptotics when considering a location-scale quantile estimator with known mean $\mu$, i.e. $q_{n, \hat{\sigma}_n}$, can be deduced from Propositions~\ref{prop-q_mu_sigma-general} and~\ref{prop-q_mu_sigma-MedianAD}. The explicit results can be found in \cite{Brautigam19_iid}. Therein one can see, for instance, that the asymptotic correlation of $q_{n, \hat{\sigma}}$ with any of the measure of dispersion estimators is independent (up to its sign) of the order $p$ of the quantile.\\[-5ex]
\item[-] Further, if we assume the location-scale distribution to be symmetric, all the asymptotic correlations have their minimum (of value $0$) at $p=0.5$ and are point symmetric around $p=0.5$.
Additionally, recall that in such a case the asymptotic correlation of the sample quantile with the MedianAD is completely independent of the choice of the underlying symmetrical location-scale distribution, see~\cite{Brautigam19_iid}.
\end{itemize}
\vspace*{-0.5cm}
Last but not least, we comment on the functions $h_1,h_2$ used in the results.
\vspace*{-0.5cm}
\begin{itemize}
\item[-] Regarding the choice of functions $h_1$ and $h_2$, some care has to be taken when applying the Delta method, as the conditions will not always be satisfied: E.g.~when using the logarithm, $h_1= \log$, the quantity $h_1^{'}(q_X(p)) = 1/q_X(p)$ is not defined at $p=F_X(0)$. In such a case, if the left-sided and right-sided limits (for the asymptotic covariance and correlation respectively) coincide, we simply set the value at the point itself, by continuity of the limit, to be the left-sided limit.\\[-4ex]
\item[-] If the functions $h_1, h_2$ are such that $h_1'(q_X(p)) \times \,h_2'(D_i) = 0$ (for any $i \in \{1,2,3\}$), then we have asymptotic linear independence in any of the results presented: The asymptotic covariances and asymptotic correlations will equal zero (as $\sgn(0)=0$, by definition).
\end{itemize}
\vspace{-3ex}
\paragraph{Conditions on the underlying distribution -}
While we specified in each theorem which moment and smoothness conditions the underlying parent random variable has to fulfil, we offer in Table~\ref{tbl-conditions-on-estimators} an overview:
For each estimator, we present separately the conditions needed for having a Bahadur/ iid-sum representation (second column) as well as the conditions for the univariate asymptotics (third column). These latter conditions are also sufficient for the joint asymptotics of any measure of dispersion and  quantile estimator, as found in the theorems.

{\small
\begin{table}[H]
\begin{center}
\parbox{460pt}{\caption{\label{tbl-conditions-on-estimators}\sf\small Conditions needed for its representation as Bahadur/iid-sum representation (second column), and for its use in computation of univariate or bivariate asymptotic normality (third column)}}\\[-1ex]
\hspace*{-0.8cm}
\begin{tabular*}{480pt}{>{\centering}m{1.8cm} |   >{\centering}m{8.7cm} >{\centering}m{5.8cm}  >{\centering}m{0.1cm}}
\tabularnewline[0.1ex] Quantile Estimator & Bahadur/ iid-sum Representation &  Asymptotic Normality of Estimator / \\ Joint Asymptotics (with a \\Measure of Dispersion Estimator) 
\tabularnewline[0.5ex] \hline 
 $q_n (p)$  &  $(\text{Q}1)$: $(C_1^{~'})$ and $(P)$ at $q_X(p)$ each & $(\text{Q}1)$ & 
\tabularnewline[2.5ex] $q_{n, \hat{\mu}, \hat{\sigma}}(p)$ \\or $q_{n, \hat{\sigma}}(p)$ & $-$ & $(\text{Q}2): \begin{cases} (M_2) \\ (X-\mu)^2 \text{~not constant} \end{cases}$ &  
\tabularnewline[0.2ex]
\tabularnewline \hline \hline
\tabularnewline[-0.3ex] Measure of Dispersion Estimator&   Bahadur / iid-sum Representation & Asymptotic Normality of Estimator / \\ Joint Asymptotics (with a \\Quantile Estimator)
\tabularnewline[0.5ex] \hline 
$\hat{\sigma}_n^2$& $-$ & $(\text{MD}1): \begin{cases} (M_2) \\ (X-\mu)^2 \text{~not constant} \end{cases}$ & 
\tabularnewline[5ex] $\hat{\theta}_n$ & $(\text{MD}2): \begin{cases} (M_1) \\ (C_0) \text{~at~} \mu \end{cases}$ & $(\text{MD}2)$ &
\tabularnewline[2.5ex] $\hat{\xi}_n$  & $(\text{MD}3): \begin{cases} (C_0) \text{~in a neighbourhood of~} \nu \pm \xi \\ (D_1) \text{~at~} \nu, \nu \pm \xi, \\ (P) \text{~at~} \nu, \text{~and at least one of~} \nu-\xi, \nu+\xi  \end{cases}$ & $(\text{MD}3)$ &
\end{tabular*}
\end{center}
\end{table}
}

We see from Table~\ref{tbl-conditions-on-estimators} that, in most cases, the conditions for the Bahadur representation and the asymptotic normality are the same. The two exceptions are the location-scale quantile estimator and the sample variance.
Also, as already mentioned, the main differences in the choice of estimators lie in the moment conditions: Using the location scale quantile requires the existence of a fourth moment (in contrast to no moment condition for the sample quantile). Also, for the measure of dispersion estimators, a fourth moment is needed when using the sample variance, whereas only a finite second moment with the sample MAD, and no moment conditions at all are imposed by the sample MedianAD. 
This is a very important remark as it needs to be taken into account when choosing estimators in practice.
Lastly, we see that for all different estimators the continuity and differentiability conditions on $F_X$ are not very restrictive. 

Apart from comparing the theoretical conditions of these asymptotics, another important point in applications is to assess the finite sample performance; this will be done in Section~\ref{ssec:sample_size_simulation_study}, via a simulation study.\\[-5ex]

\paragraph{Methods -}Let us now outline the proofs of the main results, which are given in detail in Appendix~\ref{sec:Appendix_general}.

To show the joint asymptotic normality in Theorems~\ref{th:Q-sigma}~and~\ref{thm:Q-MedianAD},  we use the bivariate central limit theorem. It means that the main work is to compute the asymptotic covariances and correlations. To ease the task, we use well chosen representations of the estimators: The Bahadur representations of the sample quantile or sample MedianAD, respectively, and a similar iid-sum representation for the sample MAD. For the sample quantile, provided first by Bahadur in \cite{Bahadur66}, we use the version of Ghosh with weaker conditions, \cite{Ghosh71}, where, under $(C_1^{~'})$ and $(P)$ at $q_X(p)$ each, it holds that
\begin{equation*} 
q_n (p) = q_X (p) + \frac{1- F_n (q_X(p)) - (1-p)}{f_X(q_X(p))} + R_{n,p}, \text{~~with~} R_{n,p} = o_P(n^{-1/2}).
\end{equation*} 
The Bahadur representation for the sample MedianAD was first shown in \cite{Hall85}; here we use the more general version with weaker conditions of \cite{Mazumder09}:
{\small
\begin{equation*}
\hat{\xi}_n - \xi = \frac{1/2 - (F_n(\nu+\xi) -F_n(\nu-\xi))}{f_X(\nu +\xi) + f_X(\nu - \xi)} - \frac{f_X(\nu +\xi) - f_X(\nu - \xi)}{f_X(\nu +\xi) + f_X(\nu - \xi)} \frac{1/2 - F_n(\nu)}{f_X(\nu)} + \Delta_n, \text{~~where~} \Delta_n= o_P(n^{-1/2}).
\end{equation*}
}
Finally, for the sample MAD (see e.g. \cite{Babu92}, \cite{Segers14}) we have that under $(C_0)$ at $\mu$
\begin{equation*}
\hat{\theta}_n = \tilde{\theta}_n + (2F_X(\mu) -1) (\bar{X}_n - \mu) + S_{n,p}, \text{~~with~} S_{n,p} = o_P(n^{-1/2}).
\end{equation*}
%
Note that in \cite{Bos13}, the authors obtained in Theorem A.1 the asymptotic distribution of sample quantiles with $r$-th absolute sample moments, in the case of an iid Gaussian sample with known mean $\mu$. We show that the extension of their result to the general iid case is straightforward. 
But some work and care is required  to extend the corresponding results (when considering the sample variance and sample MAD) to the case of an unknown, hence estimated, mean for any underlying iid distribution (with respective moment and smoothness conditions). As mentioned, the extension for arbitrary r-th absolute central sample moments (with unknown mean $\mu$) is also possible, and can be found in Appendix~\ref{sec:Appendix_extension}. 

Finally, let us mention an alternative approach to prove Theorem~\ref{th:Q-sigma};  it is based on a Taylor expansion of the sample variance (or sample MAD respectively) and the sample quantile. Although the method involves more restrictive smoothness conditions on the underlying distribution, it  has interest on its own, as a natural extension and adaption of the techniques developed by Ferguson in \cite{Ferguson99}, who considered the asymptotic joint distribution of the sample mean with the sample quantile. We present it in Appendix~\ref{sssec:appendix_taylor}.

In contrast to the work with the sample quantile, the proofs involving the parametric location-scale quantile estimator are direct computations. Thus, no specific comments on the procedure are given here and we refer to Appendix~\ref{ssec:appendix_qsigma} for details.

\vspace{-2ex}
\section{Examples} \label{sec:examples}
\vspace{-2ex}
We consider two different applications in this section. Both aim at further understanding the asymptotic dependence behaviour of the different estimators and the implications for their use in practice.
First, we bring together the asymptotic results of the historical estimation with the sample quantile and the usage of the location-scale quantile estimator: We simply compare the strength of asymptotic correlation for the different quantile and measure of dispersion estimators. 
To do so, we focus on the two examples of typical elliptical distributions, one with light tails (Gaussian), the other with heavy tails (Student).
Second, we evaluate empirically in a simulation study for elliptical distributions how well the finite sample results approximate the theoretical asymptotics of Section~\ref{sec:asympt_results}.

For the ease of readibility, by abuse of notation, the term `sample' in the context of estimators may be omitted in this section as we will be exclusively referring to sample quantities throughout: We will use variance, MAD and MedianAD synonymously for sample variance, sample MAD and sample MedianAD, respectively.

\vspace{-2ex}
\subsection{The Impact of the Choice of the Quantile Estimator for Elliptical Distributions} \label{ssec:impact_choice_of_estimator}
 \vspace{-2ex}
We consider the asymptotic correlations when using, respectively, the sample quantile and the location-scale quantile estimator assuming $\mu$ to be known.
While it is known that both estimators are consistent, the parametric location-scale quantile estimator has smaller asymptotic variance than the sample quantile. Here we look at how this influences the asymptotic correlations with the measure of dispersion estimators. Clearly, these correlations can be deduced from the theorems presented in Section~\ref{sec:asympt_results}. 
Here, for the sake of conciseness, we only present the results without derivation. 
The expressions and their computations, as well as an extended analysis also including the asymptotic covariances and ratios of asymptotic covariances or asymptotic correlations, can be found in \cite{Brautigam19_iid}. 
 
As the correlations (up to their sign) do not depend on the functions $h_1,h_2$ of the corresponding quantities, we  focus here on the case where $h_1, h_2$ are the identity functions.
This will make the results more traceable.
Further, we comment only on $p\geq 0.5$, as the case of $p < 0.5$ can be deduced by the corresponding symmetry around $p=0.5$. 

In Figure~\ref{fig:cor_norm_stud}, we plot the asymptotic correlations for the different quantile and measure of dispersion estimators ($\hat{q}_n$ and $\hat{D}_{i,n}, i=1,2,3$). The left plot corresponds to a Gaussian distribution, the two on the right to Student distributions with decreasing degrees of freedom ($\nu=10$ then $5$). 

\begin{figure}[h]
\hspace*{-1.2cm}
\begin{minipage}{0.35\textwidth}
\includegraphics[scale=1.05,width=6cm,height=6cm]{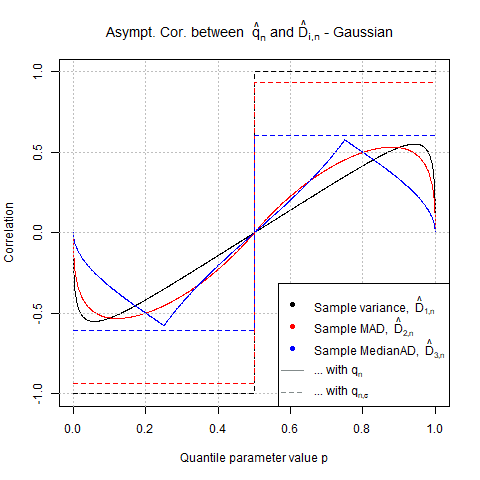}
\end{minipage}
\begin{minipage}{0.35\textwidth}
\includegraphics[scale=1.05,width=6cm,height=6cm]{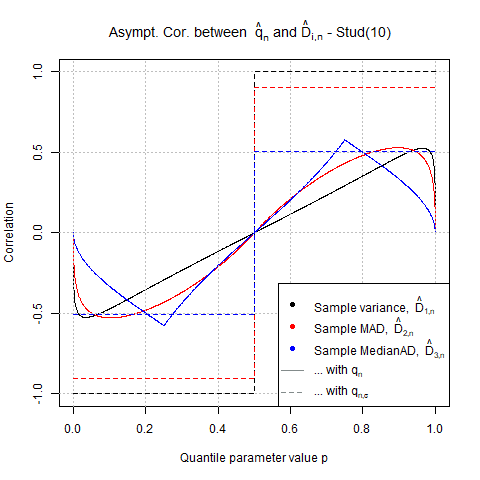}
\end{minipage}
\begin{minipage}{0.35\textwidth}
\includegraphics[scale=1.05,width=6cm,height=6cm]{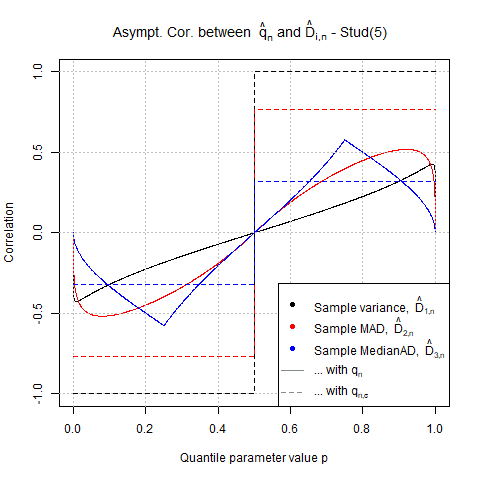}
\end{minipage}
\caption{\label{fig:cor_norm_stud} \sf\small Comparison of asymptotic correlations between the different quantile and measure of dispersion estimators;
From left to right: Gaussian, Student(10), Student(5).}
\end{figure} 
Let us start with two general observations on the plots (which we already mentioned in the remarks in Section~\ref{ssec:discussion}):
\vspace*{-0.5cm}
\begin{itemize}
\item[-] The different correlations are point symmetric around $p=0.5$ where they attain the value of $0$ (as for any symmetric location-scale distribution).\\[-4ex]
\item[-]  The correlation with the location-scale quantile estimator is constant up to the sign of $(p-0.5)$. The reason is that, in $q_{n, \hat{\sigma}}(p)$, only $\sigma$ needs to be estimated - which does not depend on the chosen order of the quantile $p$.
\end{itemize}
\vspace*{-0.5cm}
Concerning the correlations with the sample quantile $q_n$ (solid lines), we see that:
\vspace*{-0.5cm}
\begin{itemize}
\renewcommand{\labelitemi}{--}
\item[-]  All three correlations have a similar range.\\[-4ex]
\item[-]  In the tails (i.e. for small and big values of $p$), the correlations are somewhat similar for variance (black) and MAD (red), but clearly the lowest for the MedianAD (blue). This may be explained by the fact that for a very robust measure of dispersion (as e.g. the MedianAD), an extreme value in a sample does not influence the measure so much as for the variance (which incorporates every deviation from its mean to the square).\\[-4ex]
\item[-]  The distinctive shape of the correlation with the MedianAD (with its two peaks at $p=0.25$ and $p=0.75$) is related to the sample MedianAD being asymptotically equivalent to half of the sample interquartile range, $\frac{q_n(0.75) -q_n(0.25)}{2}$ for symmetric distributions (see \cite{Hall85}).\\[-4ex]
\item[-]  All three correlations with $q_n$, tend to $0$ for $p \rightarrow 1$ (but they are only defined for $p\in (0,1)$). These empirical observations are confirmed theoretically in the literature: The sample maximum, $q_n(1)$, is independent from the sample variance (see \cite{Loynes90}). Corresponding results for the MedianAD and MAD can be expected (for symmetric distributions, because of asymptotically equivalennce of sample MedianAD and half the sample IQR, the asymptotic independence can be explained by the corresponding results on order statistics, see e.g. \cite{Falk88}).
\end{itemize}
\vspace*{-0.5cm}
With respect to the location-scale quantile estimator $q_{n, \hat{\sigma}}$ (dotted lines), we notice that:
\vspace*{-0.5cm}
\begin{itemize}
\renewcommand{\labelitemi}{--}
\item[-]  The correlation is the strongest for the variance (black) and of similar magnitude as with the MAD (red). The correlation with the MedianAD (blue) is at least one third weaker.\\[-4ex]
\item[-]  For the Gaussian distribution, the correlations (dotted lines) are stronger than with the sample quantile (solid lines) whatever the measure of dispersion: The correlation with the MedianAD bounds the correlations of the different measures of dispersion with the sample quantile.
\end{itemize}
\vspace*{-0.5cm}
Finally, the following changes can be observed with heavier tailed distributions:
\vspace*{-0.5cm}
\begin{itemize}
\renewcommand{\labelitemi}{--}
\item[-]  With increasing heavyness of the tails, the correlation of the location-scale quantile estimator $q_{n, \hat{\sigma}}$ decreases significantly for $\nu=5$ in the case of the MAD, and for both $\nu=5$ and $10$ for the MedianAD.
As we are talking about the correlation of $\sigma$ with the measures of dispersion, it seems logic that the heavier the distribution, the less correlation between a robust and non-robust measure of dispersion.\\[-4ex]
\item[-]  The correlation between $q_{n, \hat{\sigma}}$ and the MedianAD does not bound the correlations with $q_n$ anymore. \\[-4ex]
\item[-]  With increasing heavyness of the tail, the correlation of the sample quantile $q_n$ with the sample variance decreases for tail values of $p$ (while staying similar for the MAD).
\end{itemize}

\subsection{The Effect of Sample Size in Estimation} 
\label{ssec:sample_size_simulation_study}
\vspace{-2ex}
We want to assess the finite sample performance, in view of the asymptotic results we obtained for the joint distribution of the quantile and dispersion measure estimators.
When working with data, we estimate the quantile and measure of dispersion estimators on finite samples of size $n$, and as well their corresponding covariance and correlation. It means we build a time-series of quantile estimates and measure of dispersion estimates, then evaluate empirically their linear dependence.

For notational convenience, we use a framework that incorporates all different cases (as introduced in Section~\ref{sec:Intro}):
Denote the quantile estimator simply by $\hat{q}_{n,t}$ for a quantile estimated at time $t$ over a sample of size $n$ (it could be either the sample quantile $q_{n,t}$ or the location-scale quantile model $q_{n, \hat{\mu}, \hat{\sigma},t}$, $q_{n, \hat{\sigma},t}$ respectively), and by $\hat{D}_{i,n,t}$, for $i=1,2,3$, the measure of dispersion of sample size $n$ (and by $D_i$ its theoretical counterpart), referring 
by $\hat{D}_{1,n,t} = \hat{\sigma}_{n,t}^2$ to the sample variance, $\hat{D}_{2,n,t}= \hat{\theta}_{n,t}$ to the sample MAD, and $\hat{D}_{3,n,t}= \hat{\xi}_{n,t}$ to the sample MedianAD.

To assess the finite sample performance, we conduct a simulation study in the following way: We simulate an iid sample with, e.g. mean $\mu=0$ and variance $\sigma^2 =1$, from three different distributions each: Either a Gaussian distribution or Student distributions with 3 and 5 degrees of freedom, respectively. The sample is of varying size $n \times l$. It is determined by the fact that we use different sample sizes $n$ for the estimation of either the quantile or the dispersion measure, with $n=126,252,504, 1008$ (being multiples or fractions of one year of data, i.e. 252 working days/ data points), and different lengths of time-series $l$ to estimate the sample correlation of interest. In each of the cases, the overall sample size needed is $n \times l$. 
Taking the example of an extreme quantile with $p=0.95$, we compute the time series of quantile estimates $\hat{q}_{n,t}(p)$ on disjoint samples and, accordingly, the time series of measure of dispersion estimates $\hat{D}_{i,n,t}$.
%
We then estimate $\Cor\left(\hat{q}_{n,t}(p), \hat{D}_{i,n,t}\right)$, using these two time series of $l$ estimates.
This procedure is repeated 1'000-fold in each case.
We report the averages of the 1'000-fold repetition with, into brackets, the corresponding empirical 95\% confidence interval. Further, we provide as benchmark the theoretical asymptotic value ($n \rightarrow \infty$, $l$ fixed) of this correlation in the last column. 
The explicit expressions in the case of a Gaussian or Student distribution of the asymptotic correlation used to calculate the theoretical values in Table~\ref{tbl:emp_cor-general-summary} could be derived from the theorems presented in Section~\ref{sec:asympt_results} (and can be found in \cite{Brautigam19_iid}).
Further, we have seen that the correlation results for location-scale distributions are independent of its parameters, hence the choice of $\mu, \sigma^2$ does not matter.

Also, we provide theoretical confidence intervals for the sample Pearson correlation coefficient (using the classical variance-stabilizing Fisher transform of the correlation coefficient for a bivariate normal distribution to compute the confidence intervals -see the original paper \cite{Fisher21} or e.g. a standard encyclopedia entry \cite{Rodriguez82}).
Note that those confidence interval values have to be considered with care. Recall that the bivariate normality of the quantile estimator and measure of dispersion estimator holds asymptotically. Hence, it is not clear if for the sample sizes $n$ considered, we can assume bivariate normality (this could be tested). Still, we provide those theoretical confidence intervals as approximate guidance.

In Table~\ref{tbl:emp_cor-general-summary}, we focus on the approximation of the joint asymptotic correlation as a function of the sample size $n$, the different dispersion estimators and the three different distributions considered.
Thus, we only consider the sample quantile (not the location-scale quantile estimator) and fix the length of the sample correlation time series to $l=50$ (from the simulations performed in \cite{Brautigam19_iid}, one can see that such a time series is long enough for a good estimation of the correlation). We present here the case for $p=0.95$. Clearly, for a higher quantile, as e.g. $p=0.99$, but with the same sample size $n$ for the estimation of the quantile, the sample correlation will be less precise (the full results of the simulation study are available in the Appendix of \cite{Brautigam19_iid}).

Recall that when working with the sample standard deviation, the existence of the fourth moment is a necessary condition. Thus, as they do not exist for a Student distribution with 3 degrees of freedom, we simply write `NA' as theoretical value instead.

Let us look at the results in Table~\ref{tbl:emp_cor-general-summary}. First we consider the Gaussian case. For the three dispersion measures, we see that a sample size of $n=126$ suffices to estimate on average the asymptotic correlation well enough. Also, the theoretical confidence intervals coming from a sample correlation of size $l=50$ are well captured by the empirical confidence intervals.
Moving to heavier tailed distributions, the picture changes a bit. The sample correlation with the sample variance does not estimate on average accurately the theoretical value. For increasing $n$, it approaches the theoretical value. This can be explained by the fact that the theoretical correlation values come from the underlying asymptotic bivariate normal distribution. Hence, for a small $n$ the corresponding sample quantities are not yet bivariate normally distributed and one would need a larger sample for this. 
This different behaviour is not observed for the MAD or MedianAD with more accurate results, comparable to the Gaussian case. While the average with the MAD is slightly (one percent point) below the theoretical value for most values of $n$ (which is acceptable), it equals the theoretical value exactly in the case of the MedianAD. In both cases, the sample confidence intervals correspond quite well to the theoretical ones, potentially indicating that the sample quantities converge faster to a bivariate normal distribution.
\begin{table}[H]
\begin{center}
\caption{\label{tbl:emp_cor-general-summary} \small Average values from a 1'000-fold repetition. Comparing the sample correlation of the sample measure of dispersion with the sample quantile, as a function of the sample size $n$ on which the quantile is estimated (fixed length $l=50$ of the time-series  used to estimate the correlation). Underlying samples are simulated from a Gaussian, Student(5) and Student(3) distributions. Average empirical values are written first (with empirical 95\% confidence interval in brackets). The theoretical asymptotic values ($n \rightarrow \infty$, $l=50$ fixed) of the sample correlation and its 95\% confidence interval, are provided as benchmark in the last column. We consider the threshold $p = 0.95$.}
\setlength{\tabcolsep}{7pt}
\begin{tabular*}{420pt}{l cc c cc}    \toprule 
$p=0.95$ & $n=126$ & $n=252$ & $n=504$ & $n=1008$ & theoretical value \\ & & & & & ($n \rightarrow \infty$) 	\\[1ex] \hline\hline
	\\ [-1.5ex]
Gaussian distr. & & & & & 
\\[1ex] $\widehat{\Cor}(\hat{\sigma}_n^2, q_n(p))$ & 55 (33;71)  & 55 (34;73) & 55 (34;73) &  55 (34;71)& 55 (32;72) 
\\[0.5ex]  $\widehat{\Cor}(\hat{\theta}_n, q_n(p))$ & 48 (26;66) & 48 (26;69) & 48 (25;69) & 48 (26;66) & 48 (23;67)
\\[0.5ex] $\widehat{\Cor}(\hat{\xi}_n, q_n(p))$ & 23 (-4;48) & 23 (-3;48) & 23 (-4;49) & 23 (-4;45) & 23~ (-5;48)\\[0.5ex]
\hline 
	\\ [-1.5ex]
Student(5) distr. & & & & & 
\\[1ex] $\widehat{\Cor}(\hat{\sigma}_n^2, q_n(p))$ & 51 (19;75) & 49 (19;71) & 47 (19;68) & 46 (20;67) & 43 (17;63)
\\[0.5ex] $\widehat{\Cor}(\hat{\theta}_n, q_n(p))$ & 50 (27;71) & 50 (27;69) & 50 (27;70) & 51 (27;69) & 51 (27;69)
\\[0.5ex] $\widehat{\Cor}(\hat{\xi}_n, q_n(p))$
& 23 (-6;50) & 23 (-6;47) & 23 (-6;47) &23 (-6;48) &  23~ (-5;48) \\[0.5ex]
\hline 
	\\ [-1.5ex]
Student(3) distr. & & & & & 
\\[1ex] $\widehat{\Cor}(\hat{\sigma}_n^2, q_n(p))$ & 25 (-8;55) & 22 (-9;52) & 19 (-9;47) & 17 (-14;44) &NA
\\[0.5ex] $\widehat{\Cor}(\hat{\theta}_n, q_n(p))$& 48 (21;68) & 47 (23;67) & 47 (20;68) & 47 (23;67) & 48 (23;67) 
\\[0.5ex] $\widehat{\Cor}(\hat{\xi}_n, q_n(p))$ & 23 (-4;49) & 22 (-7;48) & 22 (-7;47) &23 (-7;49) & 23~ (-5;48)
\\[0.5ex]
\hline
\end{tabular*}
\end{center}
\end{table}
\setlength{\tabcolsep}{6pt}
\vspace{-1ex}
In summary, the superior finite sample performance (for the heavier-tailed distributions considered) makes the use of the MAD or MedianAD more favourable than the variance. Especially keeping in mind that for heavy distributions where the fourth moment does not exist, the correlation with the variance is not defined theoretically.
As can be seen in Table~\ref{tbl:emp_cor-general-summary}, the correlation value with the MedianAD is the same for all three distributions considered (and in general for any symmetric location-scale distribution; see \cite{Brautigam19_iid}). Thus, to better discriminate the results according to the distribution, we recommend the usage of the MAD over the MedianAD.
\vspace{-4ex}

\section{Conclusion} \label{sec:conclusio}
\vspace{-2ex}
In this paper we showed the joint asymptotic bivariate normality of functions of two quantile estimators (the sample quantile and the parametric location-scale quantile estimator) and three measure of dispersion estimators (sample variance, sample MAD, sample MedianAD) each - for underlying iid models. Further, we studied theoretically the effect of using different sample sizes, 
providing an asymptotic scaling law for the asymptotic covariance and correlation. 
We also verified through simulation on elliptical distributions a good finite sample performance of the theoretical asymptotic results presented. 
All these results provide theoretical evidence for empirical observations as made in \cite{Brautigam18_WP}, \cite{Zumbach18}, \cite{Zumbach12}.

We considered two different quantile estimators, as their different speed in convergence also affects their asymptotic joint distribution with the measure of dispersion estimators, respectively. 
We analysed this in further detail by looking at the asymptotic correlations when using the sample quantile versus the location-scale quantile estimator (for each of the different measures of dispersion respectively) considering the two main elliptical distributions, the Gaussian and Student distributions. 

Apart from the speed of convergence, the use of the different quantile estimators also implies different moment conditions. So does also using the sample variance (finite fourth moment needed), versus the sample MAD (finite second moment needed) or the sample MedianAD (not even a finite mean is necessary), which may play a role in practice.

All the results should constitute an important and useful complement in the statistical literature as those estimators are either of standard use in statistics and application fields, or should become as such because of weaker conditions in the asymptotic theorems.
We are currently working on extending these joint asymptotic dependence results to different classes of stochastic processes as the class of GARCH processes, as well as developing applications of those results to risk measures and financial risk management.

\small
\bibliography{noteestimation-iid}

\begin{thebibliography}{10}

\bibitem{Babu92}
{\sc Babu, G., and Rao, C.}
\newblock Expansions for statistics involving the mean absolute deviations.
\newblock {\em Annals of the Institute of Statistical Mathematics 44}, 2
  (1992), 387--403.

\bibitem{Bahadur66}
{\sc Bahadur, R.}
\newblock A note on quantiles in large samples.
\newblock {\em The Annals of Mathematical Statistics 37}, 3 (1966), 577--580.

\bibitem{Bera16}
{\sc Bera, A., Galvao, A., Wang, L., and Xiao, Z.}
\newblock A new characterization of the normal distribution and test for
  normality.
\newblock {\em Econometric Theory 32}, 5 (2016), 1216--1252.

\bibitem{Bos13}
{\sc Bos, C., and Janus, P.}
\newblock A quantile-based realized measure of variation: New tests for
  outlying observations in financial data.
\newblock {\em Tinbergen Institute Discussion Paper 13-155/III\/} (2013).

\bibitem{Brautigam18_WP}
{\sc Br{\"a}utigam, M., Dacorogna, M., and Kratz, M.}
\newblock Predicting risk with risk measures: an empirical study.
\newblock {\em ESSEC Working Paper 1803\/} (2018).

\bibitem{Brautigam19_iid}
{\sc Br{\"a}utigam, M., and Kratz, M.}
\newblock On the dependence between quantiles and dispersion estimators.
\newblock {\em ESSEC Working Paper 1807\/} (2018).

\bibitem{Chow78}
{\sc Chow, T., and Teugels, J.}
\newblock The sum and the maximum of iid random variables.
\newblock In {\em Proceedings of the 2nd Prague Symposium on Asymptotic
  Statistics\/} (1978), pp.~81--92.

\bibitem{DasGupta06}
{\sc DasGupta, A., and Haff, L.}
\newblock Asymptotic values and expansions for the correlation between
  different measures of spread.
\newblock {\em Journal of Statistical Planning and Inference 136}, 7 (2006),
  2197--2212.

\bibitem{David98}
{\sc David, H.}
\newblock Early sample measures of variability.
\newblock {\em Statistical Science\/} (1998), 368--377.

\bibitem{Falk97}
{\sc Falk, M.}
\newblock Asymptotic independence of median and mad.
\newblock {\em Statistics \& Probability Letters 34}, 4 (1997), 341--345.

\bibitem{Falk88}
{\sc Falk, M., and Reiss, R.-D.}
\newblock Independence of order statistics.
\newblock {\em The Annals of Probability\/} (1988), 854--862.

\bibitem{Ferguson99}
{\sc Ferguson, T.}
\newblock Asymptotic joint distribution of sample mean and a sample quantile.
\newblock {\em unpublished: \url{http://www. math. ucla. edu/~
  tom/papers/unpublished/meanmed. pdf}\/} (1999).
\newblock [Online; accessed 25-April-2019].

\bibitem{Fisher21}
{\sc Fisher, R.}
\newblock On the probable error of a coefficient of correlation deduced from a
  small sample.
\newblock {\em Metron 1\/} (1921), 3--32.

\bibitem{Ghosh71}
{\sc Ghosh, J.}
\newblock A new proof of the bahadur representation of quantiles and an
  application.
\newblock {\em The Annals of Mathematical Statistics\/} (1971), 1957--1961.

\bibitem{Gorard05}
{\sc Gorard, S.}
\newblock Revisiting a 90-year-old debate: the advantages of the mean
  deviation.
\newblock {\em British Journal of Educational Studies 53}, 4 (2005), 417--430.

\bibitem{Hall85}
{\sc Hall, P., and Welsh, A.}
\newblock Limit theorems for the median deviation.
\newblock {\em Annals of the Institute of Statistical Mathematics 37}, 1
  (1985), 27--36.

\bibitem{Hampel74}
{\sc Hampel, F.}
\newblock The influence curve and its role in robust estimation.
\newblock {\em Journal of the American Statistical Association 69}, 346 (1974),
  383--393.

\bibitem{Lehmann98}
{\sc Lehmann, E., and Casella, G.}
\newblock {\em Theory of point estimation}, 2nd~ed.
\newblock Springer Science \& Business Media, 1998.

\bibitem{Lin80}
{\sc Lin, P.-E., Wu, K.-T., and Ahmad, I.}
\newblock Asymptotic joint distribution of sample quantiles and sample mean
  with applications.
\newblock {\em Communications in Statistics-Theory and Methods 9}, 1 (1980),
  51--60.

\bibitem{Loynes90}
{\sc Loynes, R.}
\newblock The variance and the range of i.i.d. random variables.
\newblock {\em Communications in Statistics - Theory and Methods 19}, 4 (1990),
  1419--1432.

\bibitem{Mazumder09}
{\sc Mazumder, S., and Serfling, R.}
\newblock Bahadur representations for the median absolute deviation and its
  modifications.
\newblock {\em Statistics \& Probability Letters 79}, 16 (2009), 1774--1783.

\bibitem{Pham01}
{\sc Pham-Gia, T., and Hung, T.}
\newblock The mean and median absolute deviations.
\newblock {\em Mathematical and Computer Modelling 34}, 7-8 (2001), 921--936.

\bibitem{Pyke65}
{\sc Pyke, R.}
\newblock Spacings.
\newblock {\em Journal of the Royal Statistical Society. Series B
  (Methodological)\/} (1965), 395--449.

\bibitem{Rodriguez82}
{\sc Rodriguez, R.}
\newblock Correlation.
\newblock In {\em Encyclopedia of Statistical Sciences, 2nd edition}, S.~Kotz,
  N.~Balakrishnan, C.~Read, B.~Vidakovic, and N.~Johnson, Eds. Wiley, New York,
  1982, pp.~1375--1385.

\bibitem{Segers14}
{\sc Segers, J.}
\newblock On the asymptotic distribution of the mean absolute deviation about
  the mean.
\newblock {\em arXiv preprint arXiv:1406.4151\/} (2014).

\bibitem{Serfling09}
{\sc Serfling, R., and Mazumder, S.}
\newblock Exponential probability inequality and convergence results for the
  median absolute deviation and its modifications.
\newblock {\em Statistics \& Probability Letters 79}, 16 (2009), 1767--1773.

\bibitem{Zumbach18}
{\sc Zumbach, G.}
\newblock Correlations of the realized volatilities with the centered
  volatility increment.
\newblock
  \url{http://www.finanscopics.com/figuresPage.php?figCode=corr_vol_r_VsDV0},
  2012.
\newblock [Online; accessed 25-April-2019].

\bibitem{Zumbach12}
{\sc Zumbach, G.}
\newblock {\em Discrete Time Series, Processes, and Applications in Finance}.
\newblock Springer Science \& Business Media, 2012.

\end{thebibliography}
\bibliographystyle{acm}

\newpage

\section*{APPENDIX}
\vspace{-1ex}
\begin{appendices} 

The appendix has three sections. The first one provides an overview of the most commonly used notations in this paper. Of big mathematical interest are the proofs of the main theorems (asymptotic distributions of quantile estimators and measure of dispersion estimators) in Appendix~\ref{sec:Appendix_general}.
Finally, Appendix~\ref{sec:Appendix_extension} presents extensions of the main asymptotic theorems, namely the general joint asymptotics of the sample quantile with the r-th absolute central sample moment for any integer $r$ and a vector-valued version of Theorem~\ref{th:Q-sigma} using more general functions. 
\vspace{-3ex}

\section{Notations used in the paper} \label{sec:Appendix_Notation_table}
\vspace{-2ex}
For convenience, we summarise the notation of the different statistical quantities with their corresponding estimators in Table~\ref{tbl-notation}.
{\small
\begin{table}[H]
\begin{center}
\parbox{440pt}{\caption{\label{tbl-notation}\sf\small Notation of statistical quantities and their (possibly various) estimators used in this paper}}\\[-1ex]
\small
\hspace*{-1.2cm} \begin{tabular*}{530pt}{p{4.1cm} p{3.3cm}  | p{4.5cm} p{5.5cm}}
\hline
\\ 
 \multicolumn{2}{c}{Statistical quantities} & \multicolumn{2}{c}{Corresponding Estimators} 
\\ [0.5ex] \hline\hline
\\
mean & $\mu$ & sample mean & $\bar{X}_n = \frac{1}{n} \sum_{i=1}^n X_i$
\\ variance & $\sigma^2$ & sample variance (unknown $\mu$) & $\hat{\sigma}_n^2 = \frac{1}{n-1} \sum_{i=1}^n (X_i - \bar{X}_n)^2$
\\ && \hspace*{2.3cm} (known $\mu$)& $\tilde{\sigma}_n^2 = \frac{1}{n} \sum_{i=1}^n (X_i - \mu)^2$
\\mean absolute deviation (MAD)&$\theta = \E[\lvert X - \mu\rvert]$ & Sample MAD (unknown $\mu$) \quad \phantom{         aa Sample MAD} (known $\mu$) & $\hat{\theta}_n = \frac{1}{n} \sum_{i=1}^n \lvert X_i - \bar{X}_n \rvert$ \phantom{   aaa manaa} $\tilde{\theta}_n = \frac{1}{n} \sum_{i=1}^n \lvert X_i - \mu \rvert$
\\ r-th centred/central moment & $\mu_r =\E[(X-\mu)^r]$ &&
\\ r-th absolute centred moment (measure of dispersion) & $m(X,r) = \E[\lvert X - \mu \rvert^r]$ & Sample measure of dispersion (unknown $\mu$)  & $\hat{m}(X,n,r) = \frac{1}{n} \sum_{i=1}^n \lvert X_i - \bar{X}_n \rvert^r$ 
\\ & & Sample measure of dispersion (known $\mu$)& $\tilde{m}(X,n,r) = \frac{1}{n} \sum_{i=1}^n \lvert X_i - \mu \rvert^r$ 
\\ cdf & $F_X(x)$ & empirical cdf & $F_n(x) = F_{n,X} (x) = \frac{1}{n} \sum_{i=1}^n \1_{(X_i \leq x)}$
\\ pdf & $f_X (x)$ &&
\\ quantile of order p & $\displaystyle q_X(p)=F_X^{-1}(p)$ & sample quantile &  $q_n (p) = X_{( \lceil np \rceil )}$
\\ & & parametric location-scale quantile estimator (unknown $\mu$) & $q_{n,\hat{\mu},\hat{\sigma}} (p) = \hat{\mu}_n + \hat{\sigma}_n q_Y (p)$
\\ & & parametric location-scale quantile estimator (known $\mu$) & $q_{n,\hat{\sigma}} (p) = \mu + \hat{\sigma}_n q_Y (p)$
\\ median & $\nu = q_X(1/2)$  & sample median & $\hat{\nu}_n = \frac{1}{2} (X_{(\lfloor \frac{n+1}{2}  \rfloor)} + X_{(\lfloor \frac{n+2}{2}  \rfloor)})$
\\ median absolute deviation (MedianAD) & $\xi = q_{\lvert X - \nu \rvert}(1/2)\quad$ where $F_{\lvert X - \nu \rvert}(x) =\quad$\; $F_X( \nu+x) - F_X (\nu - x)$ & sample MedianAD & $\hat{\xi}_n = \frac{1}{2} (W_{(\lfloor \frac{n+1}{2}  \rfloor)} + W_{(\lfloor \frac{n+2}{2}  \rfloor)})$ \quad where $W_j = \lvert X_j - \hat{\nu}_n \rvert, j=1,...,n$ 
\\ [1.5ex] \hline
\end{tabular*}
\end{center}
\end{table}
}

\vspace{-4ex}
\section{Proofs of Section~\ref{sec:asympt_results}}  
\label{sec:Appendix_general}
\vspace{-2ex}

In Appendix~\ref{ssec:appendix_qn} we cover the proofs of the asymptotics of the sample quantile with the three measure of dispersion estimators, given in Theorem~\ref{th:Q-sigma} (asymptotic distribution of sample quantile with either sample variance or sample MAD) and Theorem~\ref{thm:Q-MedianAD} (asymptotic distribution of sample quantile with the sample MedianAD). Both results will be proved using the respective Bahadur representations (Appendix~\ref{sssec:appendix_bahadur}). In the case of Theorem~\ref{th:Q-sigma}, we also offer an alternative proof via Taylor expansion (Appendix~\ref{sssec:appendix_taylor}).
Then, Appendix~\ref{ssec:appendix_qsigma} contains the corresponding proofs of the asymptotics when using the location-scale quantile estimator.  
In Appendix~\ref{ssec:appendix_samplesize} we present the proof of the scaling law (Theorem~\ref{thm:general-longer-sample}).

Note that we prove the various results first without introducing the functions $h_1, h_2$. 
Then we will consider them in the proof of Theorem~\ref{th:Q-sigma}, to give an illustration of the application of the Delta method. This latter step being common to all proofs in this section, it will be done only once.

\subsection{Proofs of Subsection~\ref{ssec:hist-estim}} \label{ssec:appendix_qn}
\vspace{-2ex}

\subsubsection{Bahadur's Method}
\label{sssec:appendix_bahadur}
\vspace{-2ex}

This approach is used to prove both Theorems~\ref{th:Q-sigma}~and~\ref{thm:Q-MedianAD}.
 


\begin{proof}{\bf of Theorem~\ref{th:Q-sigma}.}
It consists of two parts. In the first part, we assume the mean $\mu$ to be known. Using the Bahadur representation, then the bivariate Central Limit Theorem (CLT), we show the asymptotic joint normality of the sample quantile and the sample measure of dispersion with known mean, $\tilde{m}(X,n,r)$ (see Table~\ref{tbl-notation} for the notation), for any integer $r>0$. While this first part is a straightforward extension of the Gaussian case proposed in \cite{Bos13}, the second part when considering an unknown mean, involves more care.

{\sf Part 1 - Known Mean -}
\vspace{-2ex}
\begin{itemize}
\item {\it Bahadur representation}

We use the Bahadur representation for sample quantiles from an iid sample given in \cite{Ghosh71}, as the needed conditions $(C_1^{~'})$ and $(P)$ at $q_X(p)$ each are fulfilled by assumption,
\begin{equation} \label{eq:qn_Bahadur}
q_n (p) = q_X (p) + \frac{1- F_n (q_X(p)) - (1-p)}{f_X(q_X(p))} + R_{n,p}, ~~\text{where~} R_{n,p} =o_P(n^{-1/2}).
\end{equation} 
With this Bahadur representation, we are able to use the bivariate CLT for the sample quantile $q_n (p)$ and the sample measure of dispersion with known mean $\mu$, $\tilde{m}(X,n,r)$.
\item {\it Central Limit Theorem} 

Under condition $(M_r)$, $r>0$, we obtain
\begin{equation} \label{eq:CLT_bahadur_1}
 n^{-1/2} \sum_{i=1}^n \left( \begin{pmatrix} \1_{(X_i > q_X(p))} \\ \lvert X_i - \mu\rvert^r  \end{pmatrix} - \begin{pmatrix} 1-p \\ m(X,r) \end{pmatrix} \right) = n^{1/2}  \left( \begin{pmatrix} 1-F_{n}(q_X(p)) \\  \frac{1}{n} \sum_{i=1}^n \lvert X_i - \mu \rvert^r \end{pmatrix} - \begin{pmatrix} 1-p \\ m(X,r) \end{pmatrix} \right)   \overset{d}\rightarrow \mathcal{N}
(0, \hat{\Sigma}^{(r)}),
\end{equation}
where $\hat{\Sigma}^{(r)} = \begin{pmatrix} \Var(\1_{(X > q_X(p))}) & \Cov(\1_{(X > q_X(p))},\lvert X - \mu \rvert^r) \\ \Cov(\1_{(X > q_X(p))},\lvert X - \mu \rvert^r) & \Var(\lvert X - \mu \rvert^r) \end{pmatrix}$.
\\ Then, we need to pre-multiply (i.e. from the left side) equation~(\ref{eq:CLT_bahadur_1}) by $ \begin{bmatrix} 1/(f_X(q_X(p))) & 0 \\ 0 & 1 \end{bmatrix}$ to use the Bahadur representation \eqref{eq:qn_Bahadur} of the sample quantile. 
One gets (as in \cite{Bos13}, just with a different notation),
\begin{equation} \label{eq:asympt_qn_with_mu_Xnr_known_mu}
n^{1/2}  \begin{pmatrix} \frac{1-F_n(q_X(p)) - (1-p)}{f_X(q_X(p))} \\  \frac{1}{n} \sum_{i=1}^n \lvert X_i - \mu \rvert^r - m(X,r) \end{pmatrix} =
n^{1/2} \begin{pmatrix} q_n(p) - q_X(p) -R_{n,p} \\ \tilde{m} (X,n,r) -  m(X,r) \end{pmatrix}  \overset{d}\rightarrow \mathcal{N}
(0, \tilde{\Sigma}^{(r)})
\end{equation}
where now 
\begin{equation} \label{eq:Sigma_tilde_matrix}
 \tilde{\Sigma}^{(r)} = \begin{pmatrix} \frac{\Var(\1_{(X > q_X(p))})}{f_X^2(q_X(p))} & \frac{\Cov(\1_{(X > q_X(p))},\lvert X - \mu \rvert^r)}{f_X(q_X(p))} \\ \frac{\Cov(\1_{(X > q_X(p))},\lvert X - \mu \rvert^r)}{f_X(q_X(p))} & \Var(\lvert X - \mu \rvert^r) \end{pmatrix} .
\end{equation}
As $R_{n,p} = o_P(n^{-1/2})$, we can ignore it in an asymptotic analysis, as it follows from Slutsky's theorem that the distribution does not change.
Now let us compute the asymptotic covariance matrix $\tilde{\Sigma}^{(r)}$.
As we assume $(C_1)^{~'}$ and $(P)$ at $q_X(p)$, we have $F_X(q_X(p))=p$, hence
\[ \E[\1_{(X > q_X(p))}] = 1-p ~~ \text{and~} \Var(\1_{(X > q_X(p))}) = p (1-p).\] 
Therefore, introducing $\tau_k(\eta(X),p)$ defined in~\eqref{eq:def-tau}, we can write
\begin{equation} \label{eq:cov_ind_abs_moment}
\Cov(\1_{(X > q_X(p))}, \lvert X - \mu \rvert^r)=(1-p)\E[ \lvert X - \mu\rvert^r \vert X > q_X(p)]  - (1-p)\E[\lvert X - \mu\rvert^r] = \tau_r (\lvert X -\mu \rvert,p),
\end{equation}
%
which concludes to the asymptotic joint distribution of the sample quantile and the sample measure of dispersion with known $\mu$ for any integer $r>0$.
\end{itemize}

{\sf Part 2 - Unknown Mean -}

We analyse what happens with respect to the joint asymptotic distribution if we consider $\hat{m}(X,n,r)$ instead of $\tilde{m}(X,n,r)$, for $r=1,2$, and treat the two cases separately (for an extension to any integer $r>0$, see Theorem~\ref{th:qn-abs-central-moment}).
\vspace{-.2cm}
\begin{itemize}
\item Case $r=2$. Recall that $\hat{m}(X,n,2) = \frac{n}{n-1} \hat{\sigma}_n^2$ and $\tilde{m}(X,n,2) = \tilde{\sigma}_n^2$ (sample variance with known mean), so we can write, 
\begin{equation}\label{eq:rel_sigma_hat_tilde1}
\hat{m}(X,n,2) =  \frac{1}{n} \sum_{i=1}^n (X_i - \mu)^2 - (\bar{X}_n - \mu)^2
=   \tilde{m}(X,n,2) - (\bar{X}_n - \mu)^2.
\end{equation}
Since we know that $\displaystyle (\bar{X}_n - \mu) \underset{n\to\infty}{\overset{P}\rightarrow 0}$ and $\displaystyle \sqrt{n} (\bar{X}_n - \mu) \underset{n\to\infty}{\overset{d}\rightarrow} \mathcal{N}(0, \sigma^2)$, it comes, via 
Slutsky's theorem, 
\begin{equation}
\sqrt{n} (\bar{X}_n - \mu)^2 \underset{n\to\infty}{\overset{P}{\longrightarrow}} 0\,, \label{eq:rel_sigma_hat_tilde2}
\end{equation} 
from which we deduce, applying once more Slutsky's theorem, that the bivariate asymptotic distribution will not change when considering $\hat{m}(X,n,2)$ instead of $\tilde{m}(X,n,2)$.
Meaning, we have from \eqref{eq:asympt_qn_with_mu_Xnr_known_mu},~\eqref{eq:Sigma_tilde_matrix},~\eqref{eq:cov_ind_abs_moment} in the case $r=2$ that
\begin{equation*} 
\sqrt{n} \begin{pmatrix} q_n(p) - q_X(p)  \\ \hat{m}(X,n,2)  -  \sigma^2 \end{pmatrix}
\underset{n\to\infty}{\overset{d}{\longrightarrow}}  \mathcal{N}
(0, \Sigma^{(2)}),
\end{equation*}
\[ \text{where}\quad\Sigma^{(2)} = \begin{pmatrix} \frac{p(1-p)}{f_X^2(q_X(p))} & \frac{\tau_2(\lvert X - \mu \rvert,p)}{f_X(q_X(p))} \\ \frac{\tau_2(\lvert X - \mu \rvert,p)}{f_X(q_X(p))}  & \Var(\lvert X - \mu \rvert^2) \end{pmatrix}.\]
\item Case $r=1$. 
In contrast to the case $r=2$, the asymptotics of $\hat{\theta}_n= \hat{m}(X,n,1)$ and $\tilde{\theta}_n= \tilde{m}(X,n,1)$ are, in general, not the same, as we are going to see. \\
E.g. from \cite{Segers14}, for a distribution $F_X$ with finite first moment $\mu$ and continuous at $\mu$ (which is fulfilled by assumption), the sample MAD, as $n \rightarrow \infty$, satisfies almost surely
\begin{equation} \label{eq:theta_representation}
\sqrt{n} (\hat{m}(X,n,1)- \theta) = \sqrt{n} (\tilde{m}(X,n,1) - \theta) + \sqrt{n}(2 F_X(\mu) -1)  (\bar{X}_n - \mu) + o_P(1).
\end{equation}
Combining  \eqref{eq:theta_representation} with the asymptotic joint distribution obtained in part 1 (when $\mu$ is known), provides
\begin{equation*} \label{eq:asympt_qn_with_mu_Xnr}
\sqrt{n} \begin{pmatrix} q_n(p) - q_X(p)  \\ \hat{m}(X,n,1)  -  \theta \end{pmatrix}
\underset{n\to\infty}{\overset{d}{\longrightarrow}}  \mathcal{N}
(0, \Sigma^{(1)}),
\end{equation*}
\[ \text{where}\quad\Sigma^{(1)} = \begin{pmatrix} \frac{p(1-p)}{f_X^2(q_X(p))} & \displaystyle \lim_{n \rightarrow \infty} \Cov(\sqrt{n}\,\hat{q}_n(p),\sqrt{n} \,\hat{m}(X,n,1) \\ \displaystyle \lim_{n \rightarrow \infty} \Cov(\sqrt{n}\,\hat{q}_n(p),\sqrt{n} \,\hat{m}(X,n,1))  & \Var(\lvert X - \mu \rvert+ (2F_X(\mu)-1) X) \end{pmatrix},\]
as $q_n(p)$ remains unchanged (so $\Sigma_{11}^{(1)}$ is known from \eqref{eq:asympt_qn_with_mu_Xnr_known_mu}/~\eqref{eq:Sigma_tilde_matrix})  and $\Sigma_{22}^{(1)}$ follows from the representation \eqref{eq:theta_representation}). 

We conclude with the computation of the asymptotic covariance:
\begin{align*} 
 \lim_{n \rightarrow \infty} \Cov(\sqrt{n} q_n(p) , \sqrt{n} \hat{\theta}_n ) &=  \lim_{n \rightarrow \infty} \Cov(\sqrt{n} q_n(p) , \sqrt{n} \tilde{\theta}_n) +(2F_X(\mu)-1) \lim_{n \rightarrow \infty} \Cov(\sqrt{n} q_n(p) , \sqrt{n} (\bar{X}_n - \mu)) 
\\ &=  \frac{\tau_1(\lvert X-\mu \rvert, p)}{f_X(q_X(p))} + (2 F_X(\mu) -1)  \frac{\tau_1 (p)}{f_X(q_X(p))} .
\end{align*}  
In the first equality, by Slutsky's theorem, we ignored the rest term $o_P(1)$ appearing in \eqref{eq:theta_representation} as it converges to 0 in probability and thus does not change the asymptotic distribution. The first term of the second equality follows from the case with known mean $\mu$ (see \eqref{eq:Sigma_tilde_matrix}/~\eqref{eq:cov_ind_abs_moment}), and the second term $\displaystyle \lim_{n \rightarrow \infty} \Cov(\sqrt{n} \bar{X}_n, \sqrt{n} q_n (p))$ is given in \cite{Ferguson99}. 
%
\end{itemize}
Putting both cases together, we obtain, for $r=1,2$:
\begin{align*}
\lim_{n \rightarrow \infty} \Cov(\sqrt{n} \,q_n(p), \sqrt{n} \,\hat{m}(X,n,r)) &=  \frac{\tau_r(\lvert X - \mu \rvert,p) + (2-r) (2F_X(\mu)-1) \tau_1 (p)}{f_X(q_X(p))} 
\\ 
\text{and }\quad\lim_{n \rightarrow \infty} \Cor(q_n(p), \hat{m}(X,n,r)) &=  \frac{\tau_r(\lvert X - \mu \rvert,p) + (2-r) (2F_X(\mu)-1) \tau_1 (p)}{\sqrt{p(1-p)} \sqrt{\Var(\lvert X - \mu \rvert^r + (2-r) (2F_X(\mu)-1)X)}}. \numberthis \label{eq:p65-2}
\end{align*}
Note that in the specific case of $r=1,2$ we use in Theorem~\ref{th:Q-sigma} the notation $\hat{D}_{r,n}$ instead of $\hat{m}(X,n,r)$ and $D_r$ instead of $m(X,r)$.
{\sf Part 3 - Delta Method -}

Getting the expressions involving the functions $h_1, h_2$ is an application of the bivariate Delta method.
For a given function $h(x,y)= \begin{pmatrix}
h_1(x) \\ h_2(y) \end{pmatrix}$, such that the Jacobian of $h(x,y)$ exists at the point $x=q_X(p), y=m(X,r)$,
\[ J(h(q_X(p),m(X,r))) := \begin{bmatrix}
\frac{\partial h_1(x)} {\partial x} & \frac{\partial h_1(x)} {\partial y} \\ \frac{\partial h_2(y)} {\partial x} & \frac{\partial h_2(y)} {\partial y} \end{bmatrix}_{x=q_X(p), y=m(X,r)} = \begin{bmatrix}
h_1^{'}(q_X(p)) & 0 \\ 0  & h_2^{'}(m(X,r)) \end{bmatrix}, \]
we can apply the Delta method. This implies that from the asymptotics we proved above, namely
\[\sqrt{n} \, \begin{pmatrix} q_n (p) - q_X(p) \\ \hat{D}_{r,n}  - D_{r} \end{pmatrix} \; \underset{n\to\infty}{\overset{d}{\longrightarrow}} \; \mathcal{N}(0, \Sigma^{(r)}), \]
it follows that (where we denote by $z^t$ the transpose of a vector $z$)
\begin{align*}
&\sqrt{n} \, h(q_n (p), \hat{D}_{r,n} ) - h(q_X(p), D_{r})  \; \underset{n\to\infty}{\overset{d}{\longrightarrow}} \; \mathcal{N}(0, J(h(q_X(p), D_r)) \Sigma^{(r)} J(h(q_X(p), D_r))^{t}).
\end{align*}
It means to replace $\Sigma^{(r)}$ by $J(h(q_X(p), D_r)) \Sigma^{(r)} J(h(q_X(p), D_r))^{t}$. That is why the factors $h_1^{'}(q_X(p)), h_2^{'}(D_r)$ appear in the covariance terms of \eqref{eq:cov-hist-general1} and~\eqref{eq:cov-hist-general2}.
%
%
%
%
\end{proof}


\begin{proof}{\bf of Theorem~\ref{thm:Q-MedianAD}.}
Besides the Bahadur representation of the sample quantile, \eqref{eq:qn_Bahadur}, we also use a Bahadur representation (version of \cite{Mazumder09}) for the sample MedianAD $\hat{\xi}_n$, namely 
\begin{equation} \label{eq:MedianAD_bahadur}
\hat{\xi}_n - \xi = \frac{1/2 - (F_n(\nu+\xi) -F_n(\nu-\xi))}{f_X(\nu +\xi) + f_X(\nu - \xi)} - \frac{f_X(\nu +\xi) - f_X(\nu - \xi)}{f_X(\nu +\xi) + f_X(\nu - \xi)} \frac{1/2 - F_n(\nu)}{f_X(\nu)} + \Delta_n, \text{~where~} \Delta_n= o_P(n^{-1/2}).
\end{equation}
Clearly, \eqref{eq:MedianAD_bahadur} can be rewritten in terms of an iid sum as
\begin{equation} \label{eq:MedianAD_bahadur_vers2}
\hat{\xi}_n - \xi = \frac{\frac{1}{n}\sum_{i=1}^n \left(\alpha \1_{(x \leq \nu)} - f_X(\nu) \1_{(\nu - \xi < x \leq \nu + \xi)} \right)  - \frac{1}{2}\left(  \alpha -f_X(\nu) \right) }{\beta f_X(\nu)} + \Delta_n,
\end{equation}
where, for notational simplification, 
$\displaystyle \alpha :=f_X(\nu +\xi) - f_X(\nu - \xi)$ and $\displaystyle \beta := f_X(\nu +\xi) + f_X(\nu - \xi)$, respectively.

Using equations~\eqref{eq:qn_Bahadur} and~\eqref{eq:MedianAD_bahadur_vers2}, and the fact that, by definition of $\nu$ and $\xi$, $\P(X \leq \nu) =F_X(\nu) = 1/2$ and $\P( \nu -\xi < X \leq \nu + \xi) = F_{\lvert X - \nu \rvert} (\xi) = 1/2$, we apply the bivariate CLT and obtain:
\begin{align} 
 &n^{-1/2} \sum_{i=1}^n \left( \begin{pmatrix} \1_{(X_i > q_X(p))} \\ \alpha \1_{(X_i \leq \nu) }- f_X(\nu)\1_{(\nu - \xi < X_i \leq \nu + \xi)}  \end{pmatrix} - \begin{pmatrix} 1-p \\ 1/2 (\alpha-f_X(\nu)) \end{pmatrix} \right)  \notag
 \\ &= n^{1/2}  \left( \begin{pmatrix} 1-F_n(q_X(p)) \\  \frac{1}{n} \sum (\alpha \1_{(X \leq \nu) }- f_X(\nu) \1_{(\nu - \xi < X \leq \nu + \xi)} ) \end{pmatrix} - \begin{pmatrix} 1-p \\ 1/2 (\alpha-f_X(\nu)) \end{pmatrix} \right)   \overset{d}\rightarrow \mathcal{N}
(0, \tilde{\Gamma}), \label{eq:CLT_bahadur_MedianAD}
\end{align}
where $\tilde{\Gamma} = \begin{pmatrix} p(1-p) & cov_{ind: q_n, \hat{\xi}_n} \\ cov_{ind: q_n, \hat{\xi}_n} & \Var(\alpha \1_{(X_i \leq \nu) }- f_X(\nu)\1_{(\nu - \xi < X_i \leq \nu + \xi)}) \end{pmatrix}$ 
\\with $\displaystyle cov_{ind: q_n, \hat{\xi}_n}:=
\alpha \max{(0, p -1/2)} - f_X(\nu) \big(\max{(0,F_X(\nu +\xi) - \max{(F_X(\nu - \xi),p)})} - (1-p)/2 \big) $, as we are going to prove below.
\\ Then, we need to pre-multiply (i.e. from the left side) equation~(\ref{eq:CLT_bahadur_MedianAD}) by $ \begin{bmatrix} 1/(f_X(q_X(p))) & 0 \\ 0 & 1/(\beta f_X(\nu)) \end{bmatrix}$ to use the Bahadur representation of the sample quantile and of the sample MedianAD (recall \eqref{eq:qn_Bahadur}, \eqref{eq:MedianAD_bahadur_vers2}). We obtain:
\begin{equation*} 
n^{1/2} \begin{pmatrix} \frac{1-F_n(q_X(p)) - (1-p)}{f_X(q_X(p))} \\  \frac{\frac{1}{n} \sum (\alpha \1_{(X \leq \nu) }- f_X(\nu) \1_{(\nu - \xi < X \leq \nu + \xi)}) -1/2 ( \alpha - f_X(\nu))}{\beta f_X(\nu)} \end{pmatrix} =
n^{1/2} \begin{pmatrix} q_n(p) - q_X(p) -R_{n,p} \\ \hat{\xi}_n - \xi - \Delta_n \end{pmatrix}  \overset{d}{\underset{n\to\infty}{\longrightarrow}} \mathcal{N} (0, \Gamma)
\end{equation*}
where, ignoring $R_{n,p}$ and $\Delta_n$ since they are $o_P(n^{-1/2})$ a.s. (same argumentation as for $R_{n,p}$ in the proof of Theorem~\ref{th:Q-sigma}, Part 1),
\[ \Gamma = \begin{pmatrix} \displaystyle
\frac{p(1-p)}{f_X^2(q_X(p))} & \displaystyle \frac{cov_{ind: q_n, \hat{\xi}_n}}{\beta f_X(\nu) f_X(q_X(p))} \\ \displaystyle\frac{cov_{ind: q_n, \hat{\xi}_n}}{\beta f_X(\nu) f_X(q_X(p))} & \displaystyle\frac{\Var(\alpha \1_{(X_i \leq \nu) }- f_X(\nu)\1_{(\nu - \xi < X_i \leq \nu + \xi)})}{\beta^2 f_X^2(\nu)}  \end{pmatrix}. \]
%
%
We are left with computing the covariance $\displaystyle cov_{ind: q_n, \hat{\xi}_n}$ and the following variance:
\begin{align*}
\Var(\alpha &\1_{(X_i \leq \nu)} - f_X(\nu) \1_{(\nu - \xi < X_i \leq \nu + \xi)} ) 
\\ &= \alpha^2  \Var(\1_{(X_i \leq \nu)}) + f^2_X(\nu) \Var(\1_{(\nu - \xi < X_i \leq \nu + \xi)}) 
+ 2 \alpha f_X(\nu) \Cov(\1_{(X_i \leq \nu)},- \1_{(\nu - \xi < X_i \leq \nu + \xi)})
\\ & =\frac{1}{4} \left(\alpha^2 + f_X^2(\nu) - 8 \alpha f_X(\nu) (\E[\1_{(\nu - \xi < X_i \leq \nu)}] - 1/4 ) \right)
=\frac{1}{4} \left(\alpha^2 + f_X^2(\nu) - 4 \alpha f_X(\nu) (1/2 - 2 F_X(\nu- \xi) \right) 
\\&=\frac{1}{4} (f_X^2(\nu) + \gamma),\quad \text{where}\quad\gamma := \alpha^2 - 4\alpha f_X(\nu) (1- F_X(\nu -\xi) - F_X(\nu+\xi)).
\end{align*}
Let us turn to the computation of $\displaystyle cov_{ind: q_n, \hat{\xi}_n}$:
\begin{align}
\hspace*{-1cm} &\Cov(\1_{(X_i> q_X(p))}, \alpha \1_{(X_i \leq \nu) }- f_X(\nu) \1_{(\nu - \xi < X_i \leq \nu + \xi)}) =  \notag
\\&  \alpha \E[\1_{(X_i> q_X(p))} \1_{(X_i \leq \nu) }] 
- f_X(\nu) \E[ \1_{(X_i> q_X(p))} \1_{(\nu - \xi < X_i \leq \nu + \xi)}] 
 - (1-p) (\alpha - f_X(\nu))/2.  \label{eq:cov_two_ind_fct_MedianAD}
\end{align}
Let us consider one after the other the two expectations in \eqref{eq:cov_two_ind_fct_MedianAD}. 
Note that we can write (using the definition of $\nu$)
\[ \1_{(X_i> q_X(p))} \1_{(X_i \leq \nu) } = \begin{cases}
0 & \text{if } \nu \leq q_X(p) \quad (\Leftrightarrow p \geq 1/2) \\
\1_{(q_X(p) < X_i \leq \nu)} & \text{if } \nu > q_X(p) \quad (\Leftrightarrow p<1/2)\\ \end{cases}, \]
from which we deduce \; $\displaystyle \E[\1_{(X_i> q_X(p))} \1_{(X_i \leq \nu)}] = \max{(1/2 -p,0)}$.
Analogously,
\[ \1_{(X_i> q_X(p))} \1_{(\nu - \xi < X_i \leq \nu + \xi)} = \begin{cases}
0 & \text{if } q_X(p)> \nu + \xi \quad (\Leftrightarrow p>F_X(\nu+\xi)), \\
\1_{(q_X(p) < X_i \leq \nu + \xi)} & \text{if }  \nu - \xi \leq q_X(p) \leq \nu + \xi  \quad (\Leftrightarrow F_X(\nu-\xi) \leq p \leq F_X(\nu+\xi)), \\
\1_{(\nu - \xi < X_i \leq \nu + \xi)} & \text{if } q_X(p) < \nu - \xi \quad(\Leftrightarrow p< F_X(\nu - \xi)). \\ \end{cases} \]
Thus we have 
$\displaystyle \E[\1_{(X_i> q_X(p))} \1_{(\nu - \xi < X_i \leq \nu + \xi)}] = \max{(0,F_X(\nu +\xi) - \max{(F_X(\nu - \xi),p)})}$.\\
Combining these two expressions in \eqref{eq:cov_two_ind_fct_MedianAD} provides:
\begin{align*}
cov_{ind: q_n, \hat{\xi}_n} &= \alpha \max{(1/2 -p,0)} - f_X(\nu) \max{\big(0,F_X(\nu +\xi) - \max{(F_X(\nu - \xi),p)}\big)} -(1-p)(\alpha - f_X(\nu))/2
\\ &= \alpha \max{(-p/2, (p -1)/2)}  - f_X(\nu) \big( \max{(0,F_X(\nu +\xi) - \max{(F_X(\nu - \xi),p)})} - (1-p)/2 \big).
\end{align*} 
This concludes the computations. Nevertheless, to be explicit, let us write out the overall asymptotic covariance and correlation: \;$\displaystyle\lim_{n \rightarrow \infty} \Cov(\sqrt{n} q_n(p), \sqrt{n} \hat{\xi}_n) = \frac{cov_{ind:q_n, \hat{\xi}_n}}{\beta f_X(\nu) f_X(q_X(p))}=$
$$
\frac{- \max{\left(0, F_X(\nu +\xi) - \max{\left( F_X(\nu - \xi),p\right)}\right)} + \frac{1-p}{2} + \frac{f_X(\nu+ \xi) - f_X(\nu -\xi)}{f_X(\nu)} \max{\left(-\frac{p}{2},\frac{p-1}{2}\right)} }{ (f_X(\nu + \xi) + f_X(\nu-\xi)) f_X(q_X(p))}, 
$$
which is exactly the covariance in \eqref{eq:cov-hist-general2-MedianAD} for the case $h_1(x) = h_2(x) = x$ (the case with general functions $h_1,h_2$ follows directly by the application of the Delta method), whereas the correlation is as in \eqref{eq:cor-functional-sample-quant-MedianAD}:
$$
\lim_{n \rightarrow \infty} \Cor( q_n(p), \hat{\xi}_n) 
= \frac{- \max{\left(0, F_X(\nu +\xi) - \max{\left( F_X(\nu - \xi),p\right)}\right)} + \frac{1-p}{2} + \frac{f_X(\nu+ \xi) - f_X(\nu -\xi)}{f_X(\nu)} \max{\left(-\frac{p}{2},\frac{p-1}{2}\right)} } {\sqrt{\frac{p(1-p)}{4}} \sqrt{1+ \frac{\gamma}{f_X^2 ( \nu)}}}.
$$
As expected, the above computed asymptotic variance of the sample MedianAD, i.e.
\begin{equation} \label{eq:MedianAD_asympt_var}
 \lim_{n \rightarrow \infty} \Var(\sqrt{n} \hat{\xi}_n) = \frac{1+ \gamma / f_X^2(\nu)}{4\left (f_X(\nu +\xi) + f_X(\nu - \xi)\right)^2}, 
\end{equation}
exactly equals the variance of the sample MedianAD as in equation (11) of \cite{Serfling09} (while in \cite{Mazumder09} they seem to have some typos in their definition of the quantity $\gamma$ such that one does not get the same result).
\end{proof}

\vspace{-3ex}
\subsubsection{Taylor's Method}\label{sssec:appendix_taylor}
\vspace{-2ex}
As it may have interest on its own, we offer an additional proof for Theorem~\ref{th:Q-sigma}, which is based on a Taylor expansion and extends the ideas of \cite{Ferguson99}.
The proof consists of two parts. In the first and main part, we show the Taylor expansion and asymptotic normality in the case of estimating the measures of dispersion with known mean $\mu$ for any integer valued $r$ (in analogy to the first part in the proof of Theorem~\ref{th:Q-sigma}). 
The second part consists of extending the previous result to the case where we estimate the measures of dispersion in the case of an unknown mean $\mu$ for $r=1,2$ and is identical to Part 2 in the proof of Theorem~\ref{th:Q-sigma}. Therefore we focus here only on the case $\mu$ known.

We start showing the asymptotic normality in the case of estimating the measure of dispersion by $\tilde{m}(X,n,r) = \frac{1}{n} \sum_{i=1}^n \lvert X_i - \mu \rvert^r$.
This is done in three steps. The first step is to provide a representation such that our
quantities of interest, the sample quantile and the measure of dispersion estimator, are functions of the
uniform order statistics. Then, we use the Taylor expansion to prove the asymptotic normality of each
of the estimators. This is the step which requires more extensive differentiability and continuity conditions than the proof in Appendix~\ref{sssec:appendix_bahadur}. Finally, in a third step, we compute 
the covariance (and then the correlation) between the measure of dispersion estimator and the sample
quantile.

{\sf Step 1: Functions of the uniform order statistics}\\
Recall that for a standard exponentially distributed iid sample $\left( Z_1,...,Z_{n+1} \right)$, defining $U_j := \frac{\sum_{i=1}^j Z_i}{\sum_{k=1}^{n+1} Z_k}$, for $ j=1,...,n$, we have that $(U_1,...,U_{n})$ has the same distribution as the order statistics from a sample of size $n$ from a standard uniform distribution (see e.g.~\cite{Pyke65}). 
%
This allows us to express the sample quantile $q_n (p)$ and the sample measure of dispersion $\tilde{m}(X,n,r)$ as follows:
\begin{align}
q_n(p) &= X_{(\lceil{np} \rceil)} = q_X(U_{\lceil{np} \rceil}), \label{eq:qn_representation}
\\  \tilde{m}(X,n,r) &= \frac{1}{n} \sum_{i=1}^n \lvert q_X(U_i) - \mu \rvert^r. \label{eq:sample_dispersion_representation}
\end{align}
{\sf Step 2: Taylor expansions}\\
Using this, we can proceed with the Taylor expansion.
Only some work is needed for $\tilde{m}(X,n,r)$ as we can use the result of \cite{Ferguson99} for the sample quantile:
By expanding the sample quantile $q_n (p) = q_X(U_{\lceil np \rceil})$ around $p$ Ferguson gets
$\displaystyle q_n (p) = q_X(p) +  q'_X(p) (U_{\lceil np \rceil} -p) + O(n^{-2}) $.
And following equations (11), (13) and (15) in \cite{Ferguson99}, one concludes the asymptotic normality of the sample quantile
\begin{equation*} \label{eq:qn_taylor}
\sqrt{n} (q_n - q_X(p)) \underset{n\to\infty}{\sim} q'_X(p) \left(\frac{\sum_{j=1}^{\lceil np \rceil} Z_j} {\sum_{k=1}^{n+1} Z_k} -p\right) \underset{n\to \infty}{\overset{d}\longrightarrow} q'_X(p) B(p),
\end{equation*}
where $B(t) := W(t) - tW(1)$ is the Brownian bridge, $W$ denoting the standard Wiener process.
\\ Then, expanding each $q_X(U_i)$ in \eqref{eq:sample_dispersion_representation} around $i/(n+1)$, $i=1,...,n$, we obtain: $\displaystyle \tilde{m} (X,n,r) = $
\[ \frac{1}{n} \sum_{i=1}^n \left( \left\lvert q_X\left(\frac{i}{n+1}\right) - \mu \right\rvert^r  \, +\right.
\left.  r \,\left\lvert q_X\left(\frac{i}{n+1}\right) - \mu\right\rvert^{r-1} q'_X\left(\frac{i}{n+1}\right) \sgn\left(q_X\left(\frac{i}{n+1}\right) - \mu\right) \left(U_i - \frac{i}{n+1}\right) + O(n^{-2}) \right)\]
The terms of order $n^{-2}$ are negligible in the asymptotic analysis (i.e. vanish asymptotically).
\\ Then, in analogy to $\mu_n$ in \cite{Ferguson99}, we define
$\displaystyle \mu_n(X,r) := \frac{1}{n} \sum_{i=1}^n \left\lvert q_X\left(\frac{i}{n+1}\right) - \mu\right\rvert^r$. We can interpret it as the right Riemann sum: 
%
$\displaystyle \mu_n(X,r) = \frac{n+1}{n} \times\frac{1}{n+1} \sum_{i=1}^n \left\lvert q_X\left(\frac{i}{n+1}\right) - \mu \right\rvert^r \underset{n\to\infty}{\rightarrow} \int_0^1 \lvert q_X(t) - \mu \rvert^r dt$.\\
Using the transformation $t=F_X(x)$, we obtain:
\[ \int_0^1 \lvert q_X(t) - \mu \rvert^r dt = \int_{-\infty}^{+\infty} \lvert q_X(F_X(x)) - \mu \rvert^r dF_X(x) = \int_{-\infty}^{+\infty} \lvert x - \mu\rvert^r dF_X(x) =  m(X,r), \]
from which we conclude that $\displaystyle \lim_{n \rightarrow \infty} \mu_n(X,r) =  m(X,r)$.

Also, by the order of the error term of the right Riemann sum approximation, $O(n^{-1})$,
 we know that \\$\displaystyle \lim_{n \rightarrow \infty} \sqrt{n} \left(\mu_n(X,r) - m(X,r)\right) =0$.
Hence, $m(X,r)$ can be replaced by $\mu_n(X,r)$, even in asymptotics when multiplied by $\sqrt{n}$, and we can write (with the notation $\displaystyle a_n \underset{n\to\infty}{\sim}  b_n$ whenever $\displaystyle \lim_{n \rightarrow \infty} a_n/b_n =1$):
%
%
{\small
\begin{align*}
\hspace{-1cm}
\sqrt{n} (\tilde{m}(X,n,r) - \mu_n(X,r)) &\underset{n\to\infty}{\sim} \sqrt{n} \left( \frac{1}{n} \sum_{i=1}^n  r \left\lvert q_X\left(\frac{i}{n+1}\right) - \mu \right\rvert^{r-1} q'_X\left(\frac{i}{n+1}\right) \sgn\left(q_X\left(\frac{i}{n+1}\right) -\mu\right)  \left(U_i - \frac{i}{n+1}\right) \right). 
\end{align*}
}
We can then conclude to the following convergence in distribution, by using the asymptotics calculated in \cite{Ferguson99} (see eq. (12),(14) and (16) therein), 
\begin{equation*} \label{eq:sample_var_taylor}
\sqrt{n} (\tilde{m}(X,n,r) - \mu_n(X,r))  \underset{n \rightarrow \infty}{\overset{d}\rightarrow}   \int_0^1 r\, \lvert q_X(t) - \mu \rvert^{r-1} q'_X(t) \,\sgn\left(q_X (t) -\mu\right) B(t) dt, 
\end{equation*}
Hence the asymptotic normality of the measure of dispersion. 

We can now conclude the normal joint distribution by using the Cramer-Wold device (the increments of the Brownian motion being independent and normally distributed).

{\sf Step 3: Asymptotic Covariance and Correlation} 

We have, using the first two moments of the Brownian bridge,
\begin{align*}
&\lim_{n \rightarrow \infty} \Cov\left( \sqrt{n} \left(q_n(p) - q_X(p)\right), \sqrt{n} \left(\tilde{m}(X,n,r) - \mu_n(X,r)\right)\right) =\nonumber
\\&\Cov \left( q'_X(p) B(p),  \int_0^1 r \lvert q_X(t) - \mu \rvert^{r-1} q'_X(t) \sgn(q_X(t) -\mu) B(t) dt \right) \nonumber 
\\ &= q'_X(p) \int_0^1 r \lvert q_X(t) - \mu \rvert^{r-1} q'_X(t) \sgn(q_X(t) -\mu) \E[B(p) B(t)] dt  \label{eq:cov_qn_dispersion_meas_Taylor}
\\&= q'_X(p) \int_0^1 r \lvert q_X(t) - \mu \rvert^{r-1} q'_X(t) \sgn(q_X(t) -\mu) q'_X(t) ( \min{(p,t)} -pt) dt.
\end{align*}
Hence, we are left with computing the integral:
\begin{align*}
\hspace*{-1cm}
& \int_0^1 r \lvert q_X(t) - \mu \rvert^{r-1} q'_X(t) \sgn(q_X(t) -\mu) q'_X(t) (\min{(p,t)} -pt) dt 
\\&= \int_0^p r \lvert q_X(t) - \mu \rvert^{r-1} q'_X(t) \sgn(q_X(t) -\mu) q'_X(t)   t(1-p) dt
+ \int_p^1 r \lvert q_X(t) - \mu \rvert^{r-1} q'_X(t) \sgn(q_X(t) -\mu) q'_X(t) p (1-t) dt 
\\&= (1-p) \left( (\lvert q_X(t) - \mu \rvert^r t )\vert_0^p - \int_0^p \lvert q_X(t) - \mu \rvert^r dt \right) + p \left( (\lvert q_X(t) - \mu \rvert^r (1-t))\vert_p^1 + \int_p^1 \lvert q_X(t) - \mu \rvert^r dt \right) 
\\ &=  p \int_p^1 \lvert q_X(t) - \mu \rvert^r dt - (1-p) \int_0^p \lvert q_X(t) - \mu \rvert^r dt 
= p   \int_{q_X(p)}^{\infty} \lvert x- \mu \rvert^r dF_X(x) - (1-p) \int_{-\infty}^{q_X(p)}  \lvert x- \mu \rvert^r  dF_X(x) 
\end{align*}
using partial integration for each integral (with $u' =r \lvert q_X(t) - \mu \rvert^r q'_X(t) \sgn(q_X(t) -\mu) q'_X(t)$, i.e. $u= \lvert q_X(t) - \mu \rvert^r$ and $v$ being t or $1-t$ respectively) for the second equality, and $t=F_X(x)$ in the last one.
Thus, we have overall, recalling the definition of $\tau_r$ in \eqref{eq:def-tau2},
\[ \lim_{n \rightarrow \infty} \Cov( \sqrt{n} (q_n(p) - q_X(p)), \sqrt{n} (\tilde{m}(X,n,r) - m(X,r)) = q'_X(p) \tau_r( \lvert X - \mu \rvert, p) = \frac{1}{f_X(q_X(p))} \tau_r( \lvert X - \mu \rvert, p),\]
%
from which we can deduce the asymptotic correlation, namely 
\begin{align*}
\lim_{n \rightarrow \infty} \Cor(q_n(p), \tilde{m}(X,n,r)) 
= \frac{ \tau_r( \lvert X - \mu \rvert, p)}{\sqrt{p(1-p)} \sqrt{\Var(\lvert X - \mu \rvert^r)}}.
\end{align*} 


\vspace{-2ex}
\subsection{Proofs of Subsection~\ref{ssec:loc-scale-estim}} 
\label{ssec:appendix_qsigma}
\vspace{-2ex}

In the following we present the analogous proofs to Appendix~\ref{ssec:appendix_qn} but with the location-scale quantile estimator. The main task is to compute the asymptotic covariances of the respective joint asymptotic distributions.

\begin{proof}{\bf of Proposition~\ref{prop-q_mu_sigma-general}.}
Let us recall that $q_{n, \hat{\mu}, \hat{\sigma}} (p) = \bar{X}_n + q_Y(p) \tilde{\sigma}_n + o_P(1)$, and the two relations that can be deduced for $\hat{m}(X,n,r)$ from the proof of Theorem~\ref{th:Q-sigma}:
\begin{eqnarray*}
 \hat{m}(X,n,2) &:= &\frac{n-1}{n}\hat{\sigma}_n^2 = \tilde{\sigma}_n^2 +o_P(1) =: \tilde{m}(X,n,2) + o_P(1)\;\text{(obtained from~\eqref{eq:rel_sigma_hat_tilde1}, and~\eqref{eq:rel_sigma_hat_tilde2})}\\
 \hat{m}(X,n,1) &:=& \hat{\theta}_n = \tilde{\theta}_n + (2F_X(\mu)-\mu) (\bar{X}_n - \mu)  + o_P(1)=: \tilde{m}(X,n,1) + (2F_X(\mu)-\mu) (\bar{X}_n - \mu) + o_P(1) \; \text{(from~\eqref{eq:theta_representation})},
\end{eqnarray*}
which can be rewritten, for any $r=1,2$, as:
\[ \hat{m}(X,n,r) = \tilde{m}(X,n,r) + (2-r) (2F_X(\mu) -1) (\bar{X}_n - \mu) + o_P(1). \]
Since we have iid sums (and finite fourth moment by assumption), we can apply the bivariate CLT to obtain:
\begin{eqnarray*} 
 n^{1/2} \left( \begin{pmatrix} q_{n, \hat{\mu}, \hat{\sigma}}(p) \\ \hat{m}(X,n,r)  \end{pmatrix} - \begin{pmatrix} q_X(p) \\ m(X,r) \end{pmatrix} \right) &=&
 n^{1/2} \left( \begin{pmatrix} \bar{X}_n+ q_Y(p) \tilde{\sigma}_n +o_P(1) \\ \tilde{m}(X,n,r) + (2F_X(\mu)-1) (\bar{X}_n - \mu) + o_P(1)  \end{pmatrix} - \begin{pmatrix} q_X(p) \\ m(X,r) \end{pmatrix} \right) \nonumber
\\ &  \underset{n\to\infty}{\overset{d}\longrightarrow} &\mathcal{N} (0, \Lambda^{(r)})
\end{eqnarray*}
where the covariance matrix $\Lambda^{(r)}=(\Lambda^{(r)}_{ij}), i,j=1,2$, has to be determined. The component $\Lambda^{(r)}_{22}$ is already known from equation~\eqref{eq:p65-2}:
\[ \Lambda^{(r)}_{22} = \lim_{n \rightarrow \infty} \Var(\hat{m}(X,n,r)) = \sigma^{2r} \Var\left(\lvert Y \rvert^r + (2-r) (2 F_Y(0) - 1) Y\right). \]
Let us compute the other components directly. We have
\begin{align*}
\Lambda^{(r)}_{11} &= \lim_{n \rightarrow \infty} \Var(q_{n, \hat{\mu}, \hat{\sigma}}(p)) = \lim_{n \rightarrow \infty} \Var(\bar{X}_n + q_Y(p) \hat{\sigma}_n)
= \lim_{n \rightarrow \infty}  \left(\Var(\bar{X}_n) + q_Y^2(p) \Var(\hat{\sigma}_n) + 2 q_Y(p) \Cov(\bar{X}_n, \hat{\sigma}_n) \right)
\\ &= \sigma^2 + q_Y^2(p) \frac{\mu_4-\sigma^4}{(2\sigma)^2} + 2q_Y(p) \frac{\mu_3}{2\sigma}
= \sigma^2\left (1 + q_Y^2 (p) \frac{\E[Y^4]-1}{4} + q_Y(p) \E[Y^3]\right)  \numberthis \label{eq:var_q_mu_sigma_v2}
\end{align*} 
where we used the Delta method to derive from $\Var(\hat{\sigma}_n^2) = \mu_4 - \sigma^4$ and $\Cov(\bar{X}_n, \hat{\sigma}_n^2) = \mu_3$ the variance and covariance, respectively, in the case of $\hat{\sigma}_n$. We are left with
{\small
\begin{align*}
\Lambda^{(r)}_{12}=\Lambda^{(r)}_{21}& = \lim_{n \rightarrow \infty} \Cov(\sqrt{n} q_{n, \hat{\mu}, \hat{\sigma}}(p), \sqrt{n} \hat{m}(X,n,r)) 
\\ &= \lim_{n \rightarrow \infty} \Cov\left(\sqrt{n} (\bar{X}_n + q_Y(p) \hat{\sigma}_n), \sqrt{n} \left(\tilde{m}(X,n,r) + (2-r) (2F_X(\mu) -1) (\bar{X}_n - \mu)\right)\right)
\\ &= \lim_{n \rightarrow \infty} \Cov\left(\sqrt{n} (\bar{X}_n , \sqrt{n} \left(\tilde{m}(X,n,r) + (2-r) (2F_X(\mu) -1) (\bar{X}_n - \mu)\right)\right) 
\\ &\quad + q_Y(p)  \lim_{n \rightarrow \infty} \Cov\left(\sqrt{n}  \hat{\sigma}_n, \sqrt{n} \left(\tilde{m}(X,n,r) + (2-r) (2F_X(\mu) -1) (\bar{X}_n - \mu)\right)\right). \numberthis \label{eq:cov_q_mu_sigma_mu_Xnr_v2}
\end{align*} 
}
We proceed in two steps, considering separately the asymptotic covariance with the sample mean (first covariance term in \eqref{eq:cov_q_mu_sigma_mu_Xnr_v2}) and that with the sample standard deviation (second covariance term of \eqref{eq:cov_q_mu_sigma_mu_Xnr_v2}).  

In both steps we use the same techniques as in the proof of Theorem~\ref{th:Q-sigma}, when using the bivariate central limit theorem. This means, we compute the covariances by looking at the i-th element of the iid sums:
\[ \Cov( X_i, \lvert X_i - \mu \rvert^r + (2-r) (2F_X(\mu) -1) (X_i - \mu) )\;\text{and}\; 
\Cov( (X_i-\mu)^2, \lvert X_i - \mu \rvert^r + (2-r) (2F_X(\mu) -1) (X_i - \mu)). \]
{\it Step 1: Covariance with the sample mean}
\begin{align*}
 \Cov ( X_i, &\;\lvert X_i - \mu \rvert^r + (2-r) (2F_X(\mu) -1) (X_i - \mu) )
\\ &=\E[X_i \lvert X_i - \mu\rvert^r] - \mu\E[ \lvert X_i - \mu\rvert^r] + (2-r) (2F_X(\mu)-1) \Var(X_i)
\\ &= \E[ (X_i-\mu) \lvert X_i - \mu\rvert^r] + (2-r) (2F_Y(0)-1) \sigma^2
\\ &= \sigma^{r+1} \E[ \lvert Y_i \rvert^{r+1} \1_{(Y_i >0)}] -\sigma^{r+1} \E[ \lvert Y_i \rvert^{r+1} \1_{(Y_i < 0)}] + (2-r) (2F_Y(0)-1) \sigma^{2}
\\ &= \sigma^{r+1} \E[ Y_i^{r+1} ] - \sigma^{r+1} \E[ Y_i^{r+1} \1_{(Y_i < 0)} (1+ (-1)^{r+1})] + (2-r) (2F_Y(0)-1) \sigma^{r+1}
\\ &= \sigma^{r+1} \E[ Y_i^{r+1}] - \sigma^{r+1} 2 (2-r) \E[ Y_i^{r+1} \1_{(Y_i < 0)}] + (2-r) (2F_Y(0)-1) \sigma^{r+1}
\\ &= \sigma^{r+1}\left( \E[ Y_i^{r+1} ] + (2-r) \left(2F_Y(0)-1 -2\E[ Y_i^{r+1} \1_{(Y_i < 0)}]\right) \right)
\end{align*}
where, for the transformation in the last three lines, we used that, since we only consider $r=1,2$, we can write $(2-r) \sigma^2 = (2-r) \sigma^{r+1}$ and $(1+(-1)^{r+1}) = 2(2-r)$.

{\it Step 2: Covariance with the sample standard deviation}

 Considering the covariance with the sample variance, we can write
\begin{align*}
\Cov \big( (X_i-\mu)^2, &\; \lvert X_i - \mu \rvert^r + (2-r) (2F_X(\mu) -1) (X_i - \mu) \big)
\\ &= \E[\lvert X_i - \mu \rvert^{r+2}] - \sigma^2\E[\lvert X_i - \mu \rvert^{r}] +  (2-r) (2F_X(\mu) -1)\E[(X_i -\mu)^3]
\\ &= \sigma^{r+2} \left( \E[\lvert Y_i \rvert^{r+2}] - \E[\lvert Y_i \rvert^{r}] +  (2-r) (2F_Y(0) -1)\E[Y_i^3] \right)
\end{align*}
where we used in the last line the fact that $(2-r) \sigma^3 = (2-r) \sigma^{r+2}$ for $r=1,2$.

Thus, $\displaystyle  \lim_{n \rightarrow \infty} \Cov(\sqrt{n} \hat{\sigma}_n^2, \sqrt{n} \hat{m}(X,n,r)) = \sigma^{r+2} \left( \E[\lvert Y \rvert^{r+2}] - \E[\lvert Y \rvert^{r}] +  (2-r) (2F_Y(0) -1)\E[Y^3] \right) $, from which we obtain the covariance with the sample standard deviation, by applying the Delta method: 
\begin{equation*} \label{eq:cov_Sn_with_mu-Xnr}
\lim_{n \rightarrow \infty} \Cov(\sqrt{n} \hat{\sigma}_n, \sqrt{n} \hat{m}(X,n,r)) = \frac{\sigma^{r+1}}{2} \left( \E[\lvert Y \rvert^{r+2}] - \E[\lvert Y \rvert^{r}] +  (2-r) (2F_Y(0) -1)\E[Y^3] \right). 
\end{equation*}
Putting together the results of both steps in equation~\eqref{eq:cov_q_mu_sigma_mu_Xnr_v2} gives $\Lambda^{(r)}_{12}$. The asymptotic correlation follows by dividing by the asymptotic variances $\Lambda^{(r)}_{11}$ and $\Lambda^{(r)}_{22}$ computed before.
The asymptotics involving general functions $h_1,h_2$ follows by applying the Delta method.
\end{proof}

\begin{proof}{\bf of Proposition~\ref{prop-q_mu_sigma-MedianAD}.}
To use the bivariate CLT in this case, recall the Bahadur representation \eqref{eq:MedianAD_bahadur_vers2}  for the sample MedianAD, and the asymptotic equivalence of $\hat{\sigma}_n^2$ and $\tilde{\sigma}_n^2$ (see \eqref{eq:rel_sigma_hat_tilde1} and \eqref{eq:rel_sigma_hat_tilde2}). 
\\ Thus, as we have iid sums (and finite fourth moment by assumption), we can apply the bivariate CLT and obtain:
\begin{align*} 
n^{1/2} \left( \begin{pmatrix} q_{n, \hat{\mu}, \hat{\sigma}}(p) \\ \hat{\xi}_n  \end{pmatrix} - \begin{pmatrix} q_X(p) \\ \xi \end{pmatrix} \right) &=
 n^{1/2} \left( \begin{pmatrix} \bar{X}_n+ q_Y(p) \tilde{\sigma} +o_P(1) \\ \frac{\frac{1}{n}\sum_{i=1}^n \left(\alpha \1_{(x \leq \nu)} - f_X(\nu) \1_{(\nu - \xi < x \leq \nu + \xi)} \right)  - \frac{1}{2}\left(  \alpha -f_X(\nu) \right) }{\beta f_X(\nu)} + \Delta_n  \end{pmatrix} - \begin{pmatrix} q_X(p) \\ \xi \end{pmatrix} \right)  \nonumber
\\ & \underset{n\to\infty}{\overset{d}\longrightarrow}  \mathcal{N} (0, \Pi),
\end{align*}
with $\alpha =f_X(\nu +\xi) - f_X(\nu - \xi)$, $\beta = f_X(\nu +\xi) + f_X(\nu - \xi)$ and  $\Pi=(\Pi_{ij}), i,j=1,2,$ to be computed.
\\ Since $\Pi_{11}$ and $\Pi_{22}$ have been already evaluated in \eqref{eq:var_q_mu_sigma_v2} and \eqref{eq:MedianAD_asympt_var}, respectively, in the proof of Theorem~\ref{thm:Q-MedianAD},
%
%
we are left with computing the asymptotic covariance $\displaystyle \Pi_{12}=\Pi_{21}=\lim_{n \rightarrow \infty} \Cov\left(\sqrt{n} q_{n, \hat{\mu}, \hat{\sigma}}, \sqrt{n} \hat{\xi}_n\right)=$
\begin{equation}\label{eq:C12qxi}
\lim_{n \rightarrow \infty} \Cov\left(\sqrt{n} (\bar{X}_n + q_Y(p) \tilde{\sigma}_n), \sqrt{n} \frac{\frac{1}{n}\sum_{i=1}^n \left(\alpha \1_{(x \leq \nu)} - f_X(\nu) \1_{(\nu - \xi < x \leq \nu + \xi)} \right)  - \frac{1}{2}\left(  \alpha -f_X(\nu) \right) }{\beta f_X(\nu)} \right).
\end{equation} 
We proceed in two steps, looking separately at $\displaystyle \lim_{n \rightarrow \infty} \Cov(\sqrt{n} \bar{X}_n, \sqrt{n} \hat{\xi}_n)$ and $\displaystyle \lim_{n \rightarrow \infty} \Cov(\sqrt{n} \tilde{\sigma}_n^2, \sqrt{n} \hat{\xi}_n)$.
For the latter, we then use the Delta method to obtain $\displaystyle \lim_{n \rightarrow \infty} \Cov(\sqrt{n} \tilde{\sigma}_n, \sqrt{n} \hat{\xi}_n)$ instead.\\
Since we have iid sums, we are left with computing the covariance of the i-th element of the two sums each.

{\it Step 1: Covariance with the sample mean.}
Recall that $\P(X \leq \nu) = 1/2$ and $\P(\lvert X - \nu \rvert \leq \xi) = 1/2$. Then,
\begin{align*}
\hspace*{-.7cm} \Cov\left( X_i,  \frac{\alpha \1_{(x \leq \nu)} - f_X(\nu) \1_{(\nu - \xi < X_i \leq \nu + \xi)}}{\beta f_X(\nu)} \right) 
 &= \frac{1}{\beta f_X(\nu)} \left(\alpha\E[X_i \1_{(X_i \leq \nu)}] - f_X(\nu)\E[X_i \1_{(\nu - \xi < X_i \leq \nu + \xi)}] - \frac{\mu}{2}(\alpha - f_X(\nu))\right)
\\ &= \frac{\sigma}{\beta f_X(\nu)} \left(\alpha\E[Y_i \1_{(Y_i \leq \frac{\nu -\mu}{\sigma})}] - f_X(\nu)\E[Y_i \1_{(\frac{\nu - \xi -\mu}{\sigma} < Y_i \leq \frac{\nu + \xi-\mu}{\sigma})}]\right),
\end{align*}
using $X_i = \mu+ \sigma Y_i$ for the second equality.
We deduce that 
\begin{equation} \label{eq:cov_mean_xi_n}
\lim_{n \rightarrow \infty} \Cov(\sqrt{n} \bar{X}_n, \sqrt{n} \hat{\xi}_n) = 
\frac{\sigma}{\beta} \left(\frac{\alpha}{f_X(\nu)}\E[Y \1_{(Y \leq \frac{\nu -\mu}{\sigma})}] -\E[Y \1_{(\frac{\nu - \xi -\mu}{\sigma} < Y \leq \frac{\nu + \xi-\mu}{\sigma}) }]\right). 
\end{equation}  
{\it Step 2: Covariance with the sample variance.} 
\begin{align*}
& \Cov \left( (X_i-\mu)^2,  \frac{\alpha \1_{(X_i \leq \nu)} - f_X(\nu) \1_{(\nu - \xi < X_i \leq \nu + \xi)} } {\beta f_X(\nu)}\right) 
\\  &= \frac{1}{\beta f_X(\nu)} \left(\alpha\E[(X_i-\mu)^2 \1_{(X_i \leq \nu)}] - f_X(\nu)\E[(X_i-\mu)^2 \1_{(\nu - \xi < X_i \leq \nu + \xi)}] - \frac{\sigma^2}{2}(\alpha - f_X(\nu))\right)
\\ &= \frac{\sigma^2}{\beta f_X(\nu)} \left(\alpha\E[Y_i^2 \1_{(Y_i \leq \frac{\nu -\mu}{\sigma})}] - f_X(\nu)\E[Y_i^2 \1_{(\frac{\nu - \xi -\mu}{\sigma} < Y_i \leq \frac{\nu + \xi-\mu}{\sigma})}] - \frac12 (\alpha - f_X(\nu)) \right).
\end{align*}
Hence, 
\[ \lim_{n \rightarrow \infty} \Cov(\sqrt{n} \hat{\sigma}_n^2, \sqrt{n} \hat{\xi}_n) =
\frac{\sigma^2}{\beta}\left(\frac{\alpha}{f_X(\nu)}\E[Y^2 \1_{(Y \leq \frac{\nu -\mu}{\sigma})}] -\E[Y^2 \1_{(\frac{\nu - \xi -\mu}{\sigma} < Y \leq \frac{\nu + \xi-\mu}{\sigma})}] - \frac12 \left(\frac{\alpha}{f_X(\nu)} -1\right)\right) \]
and by the Delta method
\begin{equation} \label{eq:cov_Sn_with_MedianAD}
\hspace{-.9cm} \lim_{n \rightarrow \infty} \Cov(\sqrt{n} \hat{\sigma}_n, \sqrt{n} \hat{\xi}_n) =  \frac{\sigma}{2 \beta} \left(\frac{\alpha}{f_X(\nu)}\E[Y^2 \1_{(Y \leq \frac{\nu -\mu}{\sigma})}] -\E[Y^2 \1_{(\frac{\nu - \xi -\mu}{\sigma} < Y \leq \frac{\nu + \xi-\mu}{\sigma})}] - \frac12 \left(\frac{\alpha}{f_X(\nu)} -1\right) \right).
\end{equation}
Combining \eqref{eq:C12qxi}, \eqref{eq:cov_mean_xi_n} and \eqref{eq:cov_Sn_with_MedianAD} provides 
\begin{align*}
&\lim_{n \rightarrow \infty} \Cov(\sqrt{n} q_{n, \hat{\mu}, \hat{\sigma}}, \sqrt{n} \hat{\xi}_n) = \lim_{n \rightarrow \infty} \Cov(\sqrt{n} ( \bar{X}_n + q_Y(p) \hat{\sigma}_n), \sqrt{n} \hat{\xi}_n) 
\\& = \frac{\sigma}{2 \beta} \left(\frac{\alpha}{f_X(\nu)}\E[2Y \1_{(Y \leq \frac{\nu -\mu}{\sigma})}] -\E[2Y  \1_{(\frac{\nu - \xi -\mu}{\sigma} < Y \leq \frac{\nu + \xi-\mu}{\sigma})}]  \right.
\\ &\quad \left. +  q_Y(p) \left(\frac{\alpha}{f_X(\nu)}\E[Y^2 \1_{(Y \leq \frac{\nu -\mu}{\sigma})}] -\E[Y^2 \1_{(\frac{\nu - \xi -\mu}{\sigma} < Y \leq \frac{\nu + \xi-\mu}{\sigma})}] - \frac12\left(\frac{\alpha}{f_X(\nu)} - 1\right) \right) \right)
\end{align*}
from which \eqref{eq:cov_qsigma_locationscale-MedianAD} follows, by plugging in the explicit expressions for $\beta, \alpha, f_X(\nu)$ in terms of $f_Y$ (e.g. $f_X(\nu) = \frac{1}{\sigma} f_Y(\frac{\nu- \mu}{\sigma})$).
The asymptotics involving general functions $h_1, h_2$ follows by applying the Delta method.
%
%
\end{proof}

\vspace{-2ex}
\subsection{Proofs of Subsection~\ref{ssec:sample_size_theor}} \label{ssec:appendix_samplesize}
\vspace{-2ex}
The only proof in this subsection concerns Theorem~\ref{thm:general-longer-sample}. To better structure the proof, we formulate the first result of Theorem~\ref{thm:general-longer-sample}, equation~\eqref{eq:cov_longer_samplesize-general}, as a lemma and prove it separately. Using this lemma, we then prove Theorem~\ref{thm:general-longer-sample}.

As in the formulation of Theorem~\ref{thm:general-longer-sample}, for the sake of readibility, we use here in the proofs a notation for the quantile estimator $\hat{q}_n$ which only implicitly involves the order of the quantile $p$.

\begin{lemma} \label{lemma:1}
Let $v,w$ be positive integers and consider an iid sample with parent rv $X$ (with mean $\mu$ and variance $\sigma^2$, if defined). Given the respective smoothness and moment conditions (given in Theorem~\ref{thm:general-longer-sample}), the asymptotic covariance between functions of a quantile estimator $\hat{q}_{vn}$ with sample size $vn$ (be it $h_1({q}_{vn}), h_1({q}_{vn, \hat{\mu},\hat{\sigma}})$ or $h_1({q}_{vn, \hat{\sigma}})$) and the functional of the measure of dispersion estimator with sample size $wn$, $h_2(\hat{D}_{i,wn})$ for $i \in \{1,2,3 \}$, simply is
\begin{equation*} 
 \lim_{n \rightarrow \infty} \Cov\left(\sqrt{n} \, h_1(\hat{q}_{vn}), \sqrt{n}\, h_2(\hat{D}_{i,wn})\right) = 
 \frac1{\max{\left(v,w\right)}}\, \displaystyle \lim_{n \rightarrow \infty} \Cov\left(\sqrt{n} \,h_1(\hat{q}_{n}), \sqrt{n}\, h_2(\hat{D}_{i,n})\right).
\end{equation*} 
\end{lemma}
\begin{proof}{\bf of Lemma~\ref{lemma:1}.}
We want to prove the scaling law
\[ \lim_{n \rightarrow \infty} \Cov(\sqrt{n} \, h_1(\hat{q}_{vn}), \sqrt{n}\, h_2(\hat{D}_{i,wn})) =  \frac{ {\displaystyle \lim_{n \rightarrow \infty}} \Cov(\sqrt{n} \, h_1(\hat{q}_{n}), \sqrt{n}\, h_2(\hat{D}_{i,n}))}{\max{\left(v,w\right)}} \]
in the general case considering all three dispersion measure estimators ($i=1,2,3$) and the three possible quantile estimators. All the proofs in those nine different cases share a common approach that we present first (before considering each case).
\\ 
{\it General procedure -}  Consider two sequences of random variables $\displaystyle A_n = \frac{1}{n}\sum_{i=1}^n a(X_i) +a_{r,n}$ and \\$\displaystyle B_n = \frac{1}{n}\sum_{i=1}^n b(X_i) +b_{r,n}$, which are functions of $X_i$ and consist of a sum of two given parts: One linear part, an iid sum of functions $a$ and $b$, respectively, of $X_i$, denoted by $a(X_i)$ or $b(X_i)$, and a second part called the `rest' denoted by $a_{r,n}$ and $b_{r,n}$ respectively.\\
Let us compute the asymptotic covariance $\displaystyle \lim_{n \rightarrow \infty} \Cov( \sqrt{n} A_{vn}, \sqrt{n} B_{wn})$ assuming $w>v$, the reverse case being shown analogously.
%
%
We proceed in two steps. The first step consists of splitting the longer sample as $B_{wn} = B_{vn} + \text{`rest'}$ to have a covariance of equal sample size that we already wnow how to handle, as  $\displaystyle \lim_{n \rightarrow \infty} \Cov( \sqrt{n} A_{vn}, \sqrt{n} B_{vn}) = \frac{1}{v} \displaystyle \lim_{n \rightarrow \infty}   \Cov( \sqrt{n} A_{n}, \sqrt{n} B_{n})$.
The second step consists of showing why the `rest', when splitting $B_{wn}$, is negligible in the calculation of the covariance. Assuming this second step, we have
\begin{equation}\label{eq:cov_with_rests}
\hspace{-1.1cm} \lim_{n \rightarrow \infty} \Cov(\sqrt{n} A_{vn}, \sqrt{n} b_{r,wn}) = 0 \quad \text{and}\quad
\lim_{n \rightarrow \infty} \Cov(\sqrt{n} A_{vn}, \sqrt{n} b_{r,vn}) = 0,
\end{equation}
and noticing that $\displaystyle \lim_{n \rightarrow \infty} \Cov\left(\sqrt{n} A_{vn}, \frac{\sqrt{n}}{wn}\sum_{i=vn+1}^{wn}  b(X_i) \right) = 0$ because this is the covariance of iid random variables over disjoint samples, 
we can write 
\begin{align*}
&\hspace*{-0.2cm} \lim_{n \rightarrow \infty} \Cov( \sqrt{n} A_{vn}, \sqrt{n} B_{wn}) 
 = \lim_{n \rightarrow \infty} \Cov\left(\sqrt{n} A_{vn},  \sqrt{n} \left(\frac{1}{wn}\sum_{i=1}^{wn} b(X_i) +b_{r,wn}\right)\right)
= \lim_{n \rightarrow \infty} \Cov\left(\sqrt{n} A_{vn}, \sqrt{n} \frac{1}{wn}\sum_{i=1}^{wn} b(X_i)\right) \\ 
&\qquad = \lim_{n \rightarrow \infty} \Cov\left(\sqrt{n} A_{vn}, \sqrt{n} \times\frac{v}{w}\left(\frac1{vn}\sum_{i=1}^{vn} b(X_i) +b_{r,vn}\right)\right) 
 +  \lim_{n \rightarrow \infty} \Cov\left(\sqrt{n} A_{vn}, \sqrt{n} \left(\frac{1}{wn}\sum_{i=vn+1}^{wn} b(X_i) -\frac{v}{w}\, b_{r,vn} \right)\right)
\\ &\qquad= \lim_{n \rightarrow \infty} \Cov\left(\sqrt{n} A_{vn}, \sqrt{n} \times\frac{v}{w}\times B_{vn} \right)  
 =   \frac{\displaystyle \lim_{n \rightarrow \infty} \Cov(\sqrt{n} A_{n}, \sqrt{n} B_{n} )}{\max{(v,w)}}.
\end{align*}
By Cauchy-Schwarz, equations~\eqref{eq:cov_with_rests} will equal to 0 if we have that $\displaystyle \lim_{n \rightarrow\infty} \Var(\sqrt{n} \,b_{r,n}) =0$.\\
%
Thus, to show the scaling law for the quantile and measure of dispersion estimators, we will prove the following:
(i) We can express the quantile and measure of dispersion estimators in the form of $A_n, B_n$ respectively;\\ 
(ii) $\displaystyle \lim_{n \rightarrow\infty} \Var(\sqrt{n} \,b_{r,n})=0$ (covers the case $w>v$);\, (iii) $\displaystyle \lim_{n \rightarrow\infty} \Var(\sqrt{n} \,a_{r,n})=0$ (covers the case $w<v$). \\
Let us finish the proof commenting on why these six different estimators fulfil those properties, meaning that for each estimator we will show that the representation (i) is possible and that (ii) holds (which is equivalent to (iii)). \\[-5ex]
\begin{itemize}
\item {\it Sample Variance $\hat{\sigma}_n^2$.} \;
Recall from \eqref{eq:rel_sigma_hat_tilde1} that $\displaystyle\hat{\sigma}_n^2 = \frac{n}{n-1} \tilde{\sigma}_n^2 +  \frac{n}{n-1} (\bar{X}_n - \mu)^2$. Hence, we have $\frac{n-1}{n} A_n = \tilde{\sigma}_n^2$ with $a_{r,n} = \frac{n}{n-1} (\bar{X}_n - \mu)^2$ (thus, (i) holds) and we wnow, by \eqref{eq:rel_sigma_hat_tilde2} that \\
$\displaystyle \lim_{n \rightarrow \infty} \Var(\sqrt{n} (\bar{X}_n - \mu)^2) = 0$ (i.e. (ii) holds).
\item {\it Sample MAD $\hat{\theta}_n$.}\; The calculations were done in the case $r=1$ in Part 2 of the proof of Theorem~\ref{th:Q-sigma}: From \eqref{eq:theta_representation} we wnow that
$\displaystyle \hat{\theta}_n = \tilde{\theta}_n + (2 F_X(\mu) -1) (\bar{X}_n - \mu) + o_P(n^{-1/2}) $. Thus (ii) is fulfilled.
Further, recall that $\tilde{\theta}_n$ is an iid sum so (i) is fulfilled. 
\item {\it Sample MedianAD $\hat{\xi}_n$}.\;
 The Bahadur representation \eqref{eq:MedianAD_bahadur_vers2} for the sample MedianAD gives us directly (i):
\[ \hat{\xi}_n - \hat{\xi} = \frac{\frac{1}{n}\sum_{i=1}^n \left(\alpha \1_{(x \leq \nu)} - f_X(\nu) \1_{(\nu - \xi < x \leq \nu + \xi)} \right)  - \frac{1}{2}\left(  \alpha -f_X(\nu) \right) }{\beta f_X(\nu)} + \Delta_n, \;\text{~with~}\, \Delta_n \overset{a.s.}{=} o_P(n^{-1/2}).
\]
Again, as for the MAD, by the convergence to 0 in probability of $\Delta_n$, (ii) is fulfilled.
\item {\it Sample quantile $q_n$.}\;
As for the Sample MedianAD,  by the Bahadur representation  we can show that (i) and (ii) are fulfilled:  We have from \eqref{eq:qn_Bahadur}
\[ q_n (p) = q_X (p) + \frac{1- F_{n,X} (q_X(p)) - (1-p)}{f_X(q_X(p))} + R_{n,p} 
\quad\, \text{with}\quad R_{n,p}  = o_P(n^{-1/2}).
\]
Analogous to the MAD and MedianAD, this implies that for $R_{n,p}$, (ii) is fulfilled.
\item {\it Location scale quantile (wnown mean) $q_{n, \hat{\sigma}}$.} \,This case can be seen as a functional of the sample variance:  \\$q_{n, \hat{\sigma}} = q_Y(p) \sqrt{\hat{\sigma}_n^2}$. Thus, we simply apply the Delta method to the result from the case with $\hat{\sigma}_n^2$ and we are done (no need to verify (i) and (ii)).
\item {\it Location scale quantile (unwnown mean) $q_{n,\hat{\mu}, \hat{\sigma}}$.}  Recall, that $q_{n,\hat{\mu}, \hat{\sigma}} = \bar{X}_n + \hat{\sigma}_n q_{n, \hat{\sigma}}$ and $\bar{X}_n$ is already an iid sum. Thus, in comparison with the case of $q_{n,\hat{\sigma}}$, nothing changes.
\end{itemize}
For general functions $h_1, h_2$, we simply need to use the first order of the Taylor expansion (and the shown argumentation then holds for this linear approximation), as higher orders are asymptotically negligible.
\end{proof}

\begin{proof}{\bf Theorem~\ref{thm:general-longer-sample}.}
Using the result from Lemma~\ref{lemma:1}, it is straightforward to show the relation for the correlation.
We note that, by the asymptotic normality results for all the different quantile estimators and measure of dispersion estimators, we obtain, for any fixed integer $v,w>0$,
\begin{align*}
&\hspace*{-1.2cm} \lim_{n \rightarrow \infty} \Var(q_{vn}) = \lim_{n \rightarrow \infty} \frac{\Var(q_n)}{v}; ~\lim_{n \rightarrow \infty} \Var(q_{vn, \hat{\mu}, \hat{\sigma}}) = \lim_{n \rightarrow \infty} \frac{\Var(q_{n, \hat{\mu}, \hat{\sigma}})}{v}; ~\lim_{n \rightarrow \infty} \Var(q_{vn, \hat{\sigma}}) = \lim_{n \rightarrow \infty} \frac{\Var(q_{n, \hat{\sigma}})}{v};
\\ & \lim_{n \rightarrow \infty} \Var(\hat{\sigma}_{wn}^2) = \lim_{n \rightarrow \infty} \frac{\Var(\hat{\sigma}_n^2)}{w}; \quad
\lim_{n \rightarrow \infty} \Var(\hat{\theta}_{wn}) = \lim_{n \rightarrow \infty} \frac{\Var(\hat{\theta}_n)}{w}; \quad
\lim_{n \rightarrow \infty} \Var(\hat{\xi}_{wn}) = \lim_{n \rightarrow \infty} \frac{\Var(\hat{\xi}_n)}{w}
\end{align*}
from which we deduce that
\begin{align*}
&\hspace*{-0.3cm}  \lim_{n \rightarrow \infty} \Cor\left(\sqrt{n} h_1(\hat{q}_{vn}), \sqrt{n} h_2(\hat{D}_{i,wn})\right) = \lim_{n \rightarrow \infty} \frac{\Cov\left(\sqrt{n} h_1(\hat{q}_{vn}), \sqrt{n} h_2(\hat{D}_{i,wn})\right)}{\sqrt{\Var(h_1(\hat{q}_{vn}) )} \sqrt{\Var(h_2(\hat{D}_{i,wn}))}}
= \lim_{n \rightarrow \infty} \frac{ \frac{\Cov\left(\sqrt{n} h_1(\hat{q}_{n}), \sqrt{n} h_2(\hat{D}_{i,n})\right)}{\max{(v,w)}}}{\sqrt{\frac{\Var(h_1(\hat{q}_{n}))}{v}} \sqrt{\frac{\Var(h_2 (\hat{D}_{i,n}))}{w}}}
\\ &= \sqrt{\frac{vw}{\max^2{(v,w)}}} \times \lim_{n \rightarrow \infty} \Cor\left(\sqrt{n} h_1(\hat{q}_{n}), \sqrt{n} h_2(\hat{D}_{i,n})\right)  = \sqrt{\frac{\min{(v,w)}}{\max{(v,w)}}} \times \lim_{n \rightarrow \infty} \Cor\left(\sqrt{n} h_1(\hat{q}_{n}), \sqrt{n} h_2(\hat{D}_{i,n})\right),
\end{align*}	
where the second equality follows from Lemma~\ref{lemma:1} and the aforementioned scaling of the asymptotic variances.
\end{proof}

\vspace{-2ex}
\section{Extensions of Theorem~\ref{th:Q-sigma}} \label{sec:Appendix_extension}
\vspace{-2ex}

As mentioned in Subsection~\ref{ssec:hist-estim}, there are different direct extensions of Theorem~\ref{th:Q-sigma}.
First, we extend the asymptotics of the sample quantile with the sample variance or sample MAD, respectively, to the general case of the r-th absolute central sample moment for any integer $r$. This is presented in Theorem~\ref{th:qn-abs-central-moment}.
Equally, we can consider a more general function $h(x,y)$ or, as e.g. in \cite{Bos13}, we can look at the joint distribution of a vector of sample quantiles, instead of only at one sample quantile. 
These latter two ideas are combined in Theorem~\ref{th:Q-sigma-functional-vectorversion}.
Note also that we could provide in the same way extensions of Theorem~\ref{thm:Q-MedianAD} and Propositions~\ref{prop-q_mu_sigma-general} and~\ref{prop-q_mu_sigma-MedianAD}.

Let us start with the joint bivariate asymptotics of the sample quantile and the r-th absolute central sample moment.\\[-3ex]
\begin{theorem} \label{th:qn-abs-central-moment}
Consider an iid sample with parent rv $X$ having existing (unknown) mean $\mu$ and variance $\sigma^2$.
Assume conditions $(C_1^{~'}), (P)$ at $q_X(p)$ each, $(M_r)$ for the correponding integer $r$, as well as $(P)$ at $\mu$ for $r=1$. 
Then the joint behaviour of the functions $h_1$ of the sample quantile $q_n(p)$ (defined in \eqref{eq:qn}), for $p \in (0,1)$, and $h_2$ of the r-th sample absolute central moment $\hat{m}(X,n,r)$ (defined in Table~\ref{tbl-notation}), is asymptotically normal:
\begin{equation*}
\sqrt{n} \, \begin{pmatrix} h_1(q_n (p)) - h_1(q_X(p)) \\ h_2(\hat{m}(X,n,r))  - h_2(m(X,r)) \end{pmatrix} \; \underset{n\to\infty}{\overset{d}{\longrightarrow}} \; \mathcal{N}(0, \Sigma^{(r)}), 
\end{equation*}
where the asymptotic covariance matrix $\displaystyle \Sigma^{(r)}=(\Sigma^{(r)}_{ij}, 1\le i,j\le 2)$ satisfies 
\begin{align*}
\Sigma^{(r)}_{11}&=\frac{p(1-p)}{f_X^2(q_X(p))} \,\left(h_1'(q_X(p))\right)^2 ; \quad  \Sigma^{(r)}_{22}=\left(h_2'(m(X,r))\right)^2 \, \Var\left(\lvert X - \mu \rvert^r -r (X-\mu) \E[(X-\mu)^{r-1} \sgn(X-\mu)^r] \right) ;  
\\  \Sigma^{(r)}_{12}&= \Sigma^{(r)}_{21} =
h_1'(q_X(p)) \,h_2'(m(X,r)) \times \frac{\tau_r (\lvert X - \mu \rvert,p) - r \E[(X-\mu)^{r-1} \sgn(X-\mu)^r] \tau_1 (p)}{f_X(q_X(p))},  
\end{align*}
$m(X,r)$ being defined in Table~\ref{tbl-notation} and $\tau_r$ in \eqref{eq:def-tau}.

The asymptotic correlation between the functional $h_1$ of  the sample quantile and the functional $h_2$ of the r-th absolute sample moment is - up to its sign $a_{\pm} = \sgn( \,h_1'(q_X(p)) \times h_2'(m(X,r)))$ - the same whatever the choice of $h_1,h_2$:
\begin{equation*} 
\lim_{n \rightarrow \infty} \Cor\left(h_1(q_n(p)),h_2( \hat{m}(X,n,r))\right) =  a_{\pm} \times \frac{\tau_r (\lvert X - \mu \rvert,p)  - r \E[(X-\mu)^{r-1} \sgn(X-\mu)^r] \tau_1 (p)}{\sqrt{ p(1-p) \Var\left(\lvert X - \mu \rvert^r -r (X-\mu) \E[(X-\mu)^{r-1} \sgn(X-\mu)^r] \right)}}.
\end{equation*} 
\end{theorem}

The main work in the theorem is to find the asymptotics of $\hat{m}(X,n,r)= \frac{1}{n} \sum_{i=1}^n \lvert X_i - \bar{X}_n \rvert^r$ for any integer $r\geq 1$. 
As such a result might be of interest in its own right, we  give it separately in Proposition~\ref{prop:Abs_central_mom_asympt}. To prove it, we need the following Lemma first which is an adaption from Lemma 2.1 in \cite{Segers14}.
\begin{lemma} \label{lemma:segers_modif}
Consider an iid sample with parent rv $X$. Then, for any integer $v \geq 1$, given that the $v$-th moment of $X$ exists, letting $n \rightarrow \infty$, it holds almost surely that
\begin{equation} \label{eq:segers_lemma_eq}
\frac{1}{n} \sum_{i=1}^n (X_i - \mu)^{v} \left( \lvert X_i - \bar{X}_n \rvert - \lvert X_i - \mu \rvert \right) = (\bar{X}_n - \mu) \Big( \E[(X - \mu )^v \sgn(\mu-X)] +o_P(1)\Big)+ o_P(1/\sqrt{n}).
\end{equation}
\end{lemma}

\begin{proof}
The proof relying heavily on the proof of Lemma 2.1 in \cite{Segers14}, we also use the notation of \cite{Segers14}. Namely we need the rv's $A_n :=\min(\bar{X}_n, \mu), B_n := \max(\bar{X}_n, \mu)$ as well as a partition of $\{ 1,...,n \} = \mathcal{K}_n \cup \mathcal{L}_n$ with
\vspace*{-0.2cm}
\begin{align*}
\mathcal{K}_n &:= \{i=1,...,n: A_n < X_i < B_n \},
\\ \mathcal{L}_n &:= \{i=1,...,n\} \backslash \mathcal{K}_n.
\end{align*}
As, for any $x \in \R \backslash (A_n, B_n)$, $(x - \mu)^{v} \left( \lvert x - \bar{X}_n \rvert - \lvert x - \mu \rvert \right) = (x - \mu)^{v}  (\bar{X}_n - \mu) \sgn(\mu-x)$,
we rewrite the left hand side of \eqref{eq:segers_lemma_eq} (ignoring the factor $1/n$ for the moment) as
\begin{align*}
\sum_{i=1}^n (X_i - \mu)^{v} \left( \lvert X_i - \bar{X}_n \rvert - \lvert X_i - \mu \rvert \right) &= (\bar{X}_n - \mu) \sum_{i \in \mathcal{L}_n} (X_i - \mu)^v \sgn(\mu-X_i) + \sum_{i \in \mathcal{K}_n} (X_i - \mu)^v \left( \lvert X_i - \bar{X}_n \rvert - \lvert X_i - \mu \rvert \right)
\\ &= (\bar{X}_n -\mu) \sum_{i=1}^n (X_i - \mu)^v \sgn(\mu-X_i) + \tilde{R}_n, \numberthis \label{eq:another_eq}
\end{align*}
\vspace*{-0.5cm}
\begin{align*}
\text{where~} \tilde{R}_n &:= \sum_{i \in \mathcal{K}_n} (X_i - \mu)^v \left( \lvert X_i - \bar{X}_n \rvert - \lvert X_i - \mu \rvert \right) - (\bar{X}_n - \mu) \sum_{i \in \mathcal{K}_n} (X_i - \mu)^v \sgn(\mu - X_i) 
\\ &= \sum_{i \in \mathcal{K}_n} (X_i - \mu)^v \Big( \lvert X_i - \bar{X}_n \rvert - \lvert X_i - \mu \rvert - (\bar{X}_n - \mu) \sgn(\mu-X_i) \Big).
\end{align*}
Note that, by construction, for $i \in \mathcal{K}_n$, it holds that $\lvert X_i - \mu \rvert \leq \lvert \bar{X}_n - \mu \rvert$. Further, for any $x \in \R$, we have that $\lvert \sgn(x) \rvert \leq 1$ and $\lvert \lvert x- \bar{X}_n \rvert - \lvert x - \mu \rvert \rvert \leq \lvert \bar{X}_n - \mu \rvert$.
Thus, we can bound $\tilde{R}_n$ as
\begin{align*}
\lvert \tilde{R}_n \rvert \leq \lvert \mathcal{K}_n \rvert \times \lvert \bar{X}_n - \mu \rvert^v \left( \lvert \bar{X}_n - \mu \rvert + \lvert \bar{X}_n - \mu \rvert \right) = 2 \lvert \mathcal{K}_n \rvert \lvert \bar{X}_n - \mu \rvert^{v+1}
\end{align*}
where, $\lvert \mathcal{K}_n \rvert$ denotes the cardinality of the set $\mathcal{K}_n$. It is shown in \cite{Segers14} that $\frac{\lvert \mathcal{K}_n \rvert}{n} \underset{n \rightarrow \infty} \rightarrow 0$ a.s. Since we know that $\sqrt{n} \lvert \bar{X}_n - \mu \rvert^{v+1} \underset{n \rightarrow \infty}{\overset{P} \rightarrow} 0$ for any integer $v \geq 1$, we have $\sqrt{n} \frac{1}{n} \lvert \tilde{R}_n \rvert \overset{P} \rightarrow 0$, i.e. $\frac{\tilde{R}_n}{n}= o_P(1/\sqrt{n})$.

Going back to the other term in \eqref{eq:another_eq}, and multiplying it by $ \frac{1}{n}$, we obtain by the strong law of large numbers
\[ \frac{1}{n} \sum_{i=1}^n (X_i - \mu)^v \sgn(\mu-X_i) \underset{n \rightarrow \infty}{\overset{a.s.}\rightarrow} \E[(X - \mu)^v \sgn(\mu-X)], \]
hence the result.
\end{proof}

Now we are ready to state the asymptotic relation between
the r-th absolute central sample moment with known and unknown mean, respectively.
\begin{proposition} \label{prop:Abs_central_mom_asympt}
Consider an iid sample with parent rv $X$. Then, for any integer $r \geq 1$, given that the $r$-th moment of $X$ exists, it holds that
\begin{equation*} 
\sqrt{n} \left(\frac{1}{n} \sum_{i=1}^n \lvert X_i - \bar{X}_n \rvert ^r \right) \underset{n \rightarrow \infty}{\sim} \sqrt{n} \left(\frac{1}{n} \sum_{i=1}^n \lvert X_i - \mu \rvert^r \right) - r \sqrt{n} (\bar{X}_n - \mu) \E[ (X - \mu)^{r-1} \sgn(X - \mu)^r] + o_P(1) .
\end{equation*}
\end{proposition}

\begin{proof}{\bf of Propostion~\ref{prop:Abs_central_mom_asympt}}
We distinguish three different cases for $r$: Even integers $r$, $r=1$ and odd integers $r>1$.

\textit{Even integers $r$ - }
In such a case $\lvert X_i - \mu \rvert^r = (X_i - \mu)^r$ such that we can simply consider the known asymptotics for central moments.
E.g. in \cite{Lehmann98}, Example 5.2.7, they conclude that for any even integer $r>1$
\begin{equation} \label{eq:Lehmann2}
\hspace*{-1cm} \sqrt{n} \left(\frac{1}{n} \sum_{i=1}^n (X_i - \bar{X}_n)^r - \E[(X- \mu)^r] \right)  \underset{n \rightarrow \infty}{\sim} \sqrt{n} \left(\frac{1}{n} \sum_{i=1}^n (X_i - \mu)^r - \E[(X- \mu)^r] \right) - r \sqrt{n} (\bar{X}_n - \mu) \left( \frac{1}{n} \sum_{i=1}^n (X_i - \mu)^{r-1} \right) + o_P(1). 
\end{equation}
\textit{Case $r=1$ - }
This case is known too. E.g. in \cite{Segers14}, we can see that if $\E[X] < \infty$ and $F_X$ is continuous at $\mu$, it holds that, as $n \rightarrow \infty$, almost surely,
\begin{equation} \label{eq:theta_representation2}
\sqrt{n} \left(\frac{1}{n} \sum_{i=1}^n \lvert X_i - \bar{X}_n \rvert - \E[\lvert X -\mu \rvert] \right) = \sqrt{n} \left(\frac{1}{n} \sum_{i=1}^n \lvert X_i - \mu  \rvert - \E[\lvert X -\mu \rvert] \right) + \sqrt{n}(2 F_X(\mu) -1)  (\bar{X}_n - \mu) + o_P(1).
\end{equation}
%
\textit{Odd integer $r>1$ - }
This is the only case requiring some work. Set $r=2u+1$, for any integer $u\geq 1$.

By binomial expansion we get
\begin{align*}
 \sum_{i=1}^n \lvert X_i - \bar{X}_n \rvert^{2u+1} &= \sum_{i=1}^n ( X_i - \bar{X}_n )^{2u} \lvert X_i - \bar{X}_n \rvert
=  \sum_{k=0}^{2u} (-1)^k \binom{2u}{k} (\bar{X}_n - \mu)^k \left( \sum_{i=1}^n (X_i - \mu)^{2u-k} \lvert X_i - \bar{X}_n \rvert \right)
\\ &=  \sum_{i=1}^n \lvert X_i - \bar{X}_n \rvert (X_i - \mu)^{2u}  -2u (\bar{X}_n - \mu)  \sum_{i=1}^n  \lvert X_i - \bar{X}_n \rvert (X_i - \mu)^{2u-1} 
\\ & \quad \quad + \sum_{k=2}^{2u}  (-1)^k \binom{2u}{k} (\bar{X}_n - \mu)^k \left( \sum_{i=1}^n (X_i - \mu)^{2u-k} \lvert X_i - \bar{X}_n \rvert \right). \numberthis \label{eq:temp_eq}
\end{align*}
Recall that, for the asymptotics, we need to multiply \eqref{eq:temp_eq} by $\sqrt{n} \frac{1}{n}$. Following the analogous argumentation as in \cite{Lehmann98} for even integers, we conclude that all terms in \eqref{eq:temp_eq} apart from the first two ($k=0,1$) vanish as $\sqrt{n} (\bar{X}_n -\mu)^v \overset{P}\rightarrow 0$ for $v\geq 2$.
Hence, we are left with the analysis of the first two terms of \eqref{eq:temp_eq}. 
\\ For the first two terms, multiplied by $1/n$, we have, as $n\rightarrow \infty$,
\begin{align*} 
 \frac{1}{n} \sum_{i=1}^n (X_i - \mu)^{2u} \lvert X_i - \bar{X}_n \rvert &= \frac{1}{n} \sum_{i=1}^n \left( \lvert X_i - \mu \rvert^{2u+1} + (X_i - \mu)^{2u} \left( \lvert X_i - \bar{X}_n \rvert - \lvert X_i - \mu \rvert \right) \right)
\\&= \frac{1}{n} \sum_{i=1}^n \lvert X_i - \mu \rvert^{2u+1} + (\bar{X}_n - \mu) \Big( \E[(X - \mu )^{2u} \sgn(\mu-X)] +o_P(1) \Big)+ o_P(\frac{1}{\sqrt{n}}), \numberthis \label{eq:term1}
\end{align*}
\vspace*{-1cm}
\begin{align*}
\text{and~} \frac{1}{n} (-2u) (\bar{X}_n - \mu) &\sum_{i=1}^n (X_i - \mu)^{2u-1} \lvert X_i - \bar{X}_n \rvert 
\\ &=  \frac{1}{n} (-2u) (\bar{X}_n - \mu) \sum_{i=1}^n \left( (X_i - \mu)^{2u} \sgn(X_i - \mu)  
+  (X_i - \mu)^{2u-1} \left(\lvert X_i - \bar{X}_n \rvert - \lvert X_i - \mu \rvert \right) \right).
\\ &\overset{a.s.}{=} -2u(\bar{X}_n - \mu) \E[(X - \mu )^{2u} \sgn(X-\mu)]  \numberthis \label{eq:term2}
\\ &-2u(\bar{X}_n - \mu) \left( (\bar{X}_n - \mu) (\E[(X - \mu )^{2u-1} \sgn(\mu-X)] +o_P(1))+ o_P(\frac{1}{\sqrt{n}})  \right), 
\end{align*}
applying Lemma~\ref{lemma:segers_modif} and the law of large numers in \eqref{eq:term1} and \eqref{eq:term2}.  Putting \eqref{eq:term1} and \eqref{eq:term2} together, we get from \eqref{eq:temp_eq}, for $n \rightarrow \infty$, that a.s.
\begin{align*}
\frac{1}{n} \sum_{i=1}^n \lvert X_i - \bar{X}_n \rvert^{2u+1} &= \frac{1}{n} \sum_{i=1}^n \lvert X_i - \mu \rvert^{2u+1} - (2u+1) (\bar{X}_n - \mu) \Big(\E[(X - \mu )^{2u} \sgn(X-\mu)] +o_P(1)\Big) + o_P(\frac{1}{\sqrt{n}})
\\ &-2u(\bar{X}_n - \mu) \left( (\bar{X}_n - \mu) \Big(\E[(X - \mu )^{2u-1} \sgn(\mu-X)] +o_P(1)\Big)+ o_P(\frac{1}{\sqrt{n}})  \right) .
\end{align*}
As $\sqrt{n} (\bar{X}_n - \mu)^2 \overset{P}\rightarrow 0$, we can conclude to the following asymptotic equivalence:
\begin{equation} \label{eq:abs_central_sample_mom_uneven}
\sqrt{n} \frac{1}{n} \sum_{i=1}^n \lvert X_i - \bar{X}_n \rvert^{2u+1} \underset{n \rightarrow \infty}{\sim} \sqrt{n} \frac{1}{n} \sum_{i=1}^n \lvert X_i - \mu \rvert^{2u+1} - (2u+1) \sqrt{n} (\bar{X}_n - \mu) \E[(X - \mu )^{2u} \sgn(X-\mu)]. 
\end{equation}
%
{\it A statement for any integer $r$.}
To conclude, we can summarize the different cases in~\eqref{eq:Lehmann2},~\eqref{eq:theta_representation2} and~\eqref{eq:abs_central_sample_mom_uneven} as follows. For any integer $r \geq 1$, it holds 
\begin{equation*}
\hspace*{-.3cm} \sqrt{n} \left(\frac{1}{n} \sum_{i=1}^n \lvert X_i - \bar{X}_n \rvert ^r \right)  \underset{n \rightarrow \infty}{\sim} \sqrt{n} \left(\frac{1}{n} \sum_{i=1}^n \lvert X_i - \mu \rvert^r \right) - r \sqrt{n} (\bar{X}_n - \mu) \E[(X - \mu)^{r-1} \sgn(X-\mu)^r] + o_P(1). 
\end{equation*}
\end{proof}

\begin{proof}{\bf of Theorem~\ref{th:qn-abs-central-moment}}
Recall from Part 1 of the proof of Theorem~\ref{th:Q-sigma}, that 
\[ \lim_{n \rightarrow \infty} \Cov(\sqrt{n} q_n(p), \sqrt{n} \tilde{m}(X,n,r) ) = \frac{\tau_r(\lvert X - \mu \rvert,p)}{f_X(q_X(p))}. \]
Further, from \cite{Ferguson99}, we know that $\lim_{n \rightarrow \infty} \Cov(\sqrt{n}~q_n(p), \sqrt{n} (\bar{X}_n - \mu) ) = \frac{\tau_1(p)}{f_X(q_X(p))}.$
Hence, using Proposition~\ref{prop:Abs_central_mom_asympt},
we can write
\begin{align*}
\Sigma_{12}^{(r)} = \Sigma_{21}^{(r)} &=  \lim_{n \rightarrow \infty} \Cov(\sqrt{n} ~q_n(p), \sqrt{n}~ \hat{m}(X,n,r) ) 
\\ &=  \lim_{n \rightarrow \infty} \Cov(\sqrt{n} ~q_n(p),  \sqrt{n}~ \tilde{m}(X,n,r) - r \sqrt{n} (\bar{X}_n - \mu) \E[ (X - \mu)^{r-1} \sgn(X - \mu)^r] + o_P(1) ) 
\\ &= \frac{\tau_r(\lvert X - \mu \rvert,p)}{f_X(q_X(p))} - \frac{r  \E[ (X - \mu)^{r-1} \sgn(X - \mu)^r] \tau_1(p)}{f_X(q_X(p))}.
\end{align*}
$\Sigma_{11}^{(r)}$ remains unchanged in comparison to Theorem~\ref{th:Q-sigma}, and the variance of $\hat{m}(X,n,r)$ follows directly from Proposition~\ref{prop:Abs_central_mom_asympt}:
\begin{align*}
\Sigma_{22}^{(r)} &= \lim_{n \rightarrow \infty} \Var(\sqrt{n} ~\hat{m}(X,n,r)) 
\\ &=\lim_{n \rightarrow \infty} \Var( ~\sqrt{n} \tilde{m}(X,n,r) - r \sqrt{n} (\bar{X}_n - \mu) \E[ (X - \mu)^{r-1} \sgn(X - \mu)^r] + o_P(1) ) )
\\ &= \Var( \lvert X - \mu \rvert^r - r \E[ (X - \mu)^{r-1} \sgn(X - \mu)^r] (X- \mu)),
\end{align*}
which concludes the proof of the theorem.
\end{proof}

The second extension considers a more general function $h(x,y)$ and a vector of sample quantiles, instead of only one sample quantile. For this, denote by $\textbf{q}_X(\textbf{p})$ the m-vector of quantiles evaluated at $p_i, i=1,...,m$, where $0<p_1 <...< p_m<1$, and by $\textbf{q}_n(\textbf{p})$ the corresponding $m$-vector of sample quantiles $q_n(p_i), i=1,...,m$. Recall that $z^{t}$ denotes the transpose of a vector $z$.

We state the theorem for $r=1,2$ but it could be formulated for any integer $r>0$ as for Theorem~\ref{th:qn-abs-central-moment}.
\begin{theorem} \label{th:Q-sigma-functional-vectorversion}
Consider an iid sample with parent rv $X$ having existing (unknown) mean $\mu$ and variance $\sigma^2$. Further, consider a function $h: \R^{m+1} \mapsto \R^{m+1}$, i.e. $h(x_1,...,x_m,y) = \begin{pmatrix}
h_1 (x_1,...,x_m,y) \\... \\ h_{m+1}(x_1,...,x_m,y)
\end{pmatrix} $ with continuous real-valued components $h_i(x_1,...,x_m,y), i=1,...,m$, and existing partial derivatives denoted by $\partial_i h_j, i,j \in \{1,...,m+1\}$. Assume conditions $(C_1^{~'}), (P)$ at 
$q_X(p_i), i=1,...,m$ each, and $(M_r)$ for $r=1,2$, respectively, as well as $(P)$ at $\mu$ for $r=1$. 
Then, the joint behaviour of the functional $h$ of the sample quantile vector $\textbf{q}_n(\textbf{p})$ and of the measure of dispersion $\hat{m}(X,n,r)$ (defined in Table~\ref{tbl-notation}) is asymptotically normal:
\begin{equation*} 
\sqrt{n} \, h \begin{pmatrix} \textbf{q}_n (\textbf{p}) \\ \hat{m}(X,n,r) \end{pmatrix} -  h \begin{pmatrix}
\textbf{q}_X(\textbf{p}) \\  m(X,r) \end{pmatrix} \; \underset{n\to\infty}{\overset{d}{\longrightarrow}} \; \mathcal{N}(0, J(h\begin{pmatrix} \textbf{q}_n (\textbf{p}) \\ \hat{m}(X,n,r) \end{pmatrix}) \Sigma^{(m,r)} J( h\begin{pmatrix} \textbf{q}_n (\textbf{p}) \\ \hat{m}(X,n,r) \end{pmatrix})^t), 
\end{equation*}
where $ J( h\begin{pmatrix} \textbf{q}_n (\textbf{p}) \\ \hat{m}(X,n,r) \end{pmatrix})$ is the Jacobian matrix of $h\begin{pmatrix} \textbf{q}_n (\textbf{p}) \\ \hat{m}(X,n,r) \end{pmatrix}$ and the asymptotic covariance matrix $\Sigma^{(m,r)}$ of dimension $(m+1) \times (m+1)$ can be written as 
\begin{equation*} \label{eq:cov_matrix-q-sigma-vector}
 \Sigma^{(m,r)}= \begin{bmatrix} \Sigma^{(m)} & \textbf{s}(X,r) \\ \textbf{s}(X,r)^t  &\Var\Big( (\lvert X - \mu \rvert^r + (2-r)(2F_X(\mu)-1)X \Big) \end{bmatrix}
\end{equation*}
with $\Sigma^{(m)}_{ij} = \Sigma^{(m)}_{ji} = \frac{p_i (1-p_j)}{f_X(q_X(p_i)) f_X(q_X(p_j))}$ for $i,j \in \{ 1,...,m\}$ and the i-th element of the m-vector $\textbf{s}(X,r)$ being
\\ $\frac{\tau_r (\lvert X - \mu \rvert,p_i) +(2-r)(2F_X(\mu)-1) \tau_1 (p_i)}{f_X(q_X(p_i))}, i=1,...,m$, where $\tau_r$ is defined in \eqref{eq:def-tau}.
\end{theorem}
As a corollary of the theorem, we can state how the result explicitly looks like if we go back to the one-dimensional sample quantile case, considering $r=1,2$, as in Theorem~\ref{th:Q-sigma} but with a general function $h(x,y)$.
\begin{corollary}
Consider an iid sample with parent rv $X$ having existing (unknown) mean $\mu$, variance $\sigma^2$ and a function $h(x,y) = \begin{pmatrix}
h_1 (x,y) \\ h_2(x,y)
\end{pmatrix} $ with continuous real-valued components $h_1(x,y), h_2(x,y)$ and existing partial derivatives denoted by $\partial_i h_j, i,j \in \{1,2\}$. Assume conditions $(C_1^{~'}), (P)$ at 
$q_X(p)$ each, and $(M_r)$ for $r=1,2$ respectively as well as $(P)$ at $\mu$ for $r=1$.
Then, the joint behaviour of the functional $h$ of the sample quantile $q_n(p)$ (for $p \in (0,1)$) and of the measure of dispersion $\hat{m}(X,n,r)$ (defined in Table~\ref{tbl-notation}) is asymptotically normal:
$$
\sqrt{n} \, h\begin{pmatrix} q_n (p)  \\ \hat{m}(X,n,r)   \end{pmatrix} -  h\begin{pmatrix} \hat{q}_X (p)  \\ m(X,r)   \end{pmatrix} \; \underset{n\to\infty}{\overset{d}{\longrightarrow}} \; \mathcal{N}(0, \Sigma^{(h,r)}), 
$$
where the asymptotic covariance matrix $\displaystyle \Sigma^{(h,r)}=(\Sigma^{(h,r)}_{ij}, 1\le i,j\le 2)$ satisfies, denoting by abuse of notation, $ h_1 = h_1 (q_X(p),m(X,r))$ and $h_2 = h_2 (q_X(p),m(X,r))$,
\begin{align*}
\Sigma^{(h,r)}_{11}&=\Var(q_n (p))\, \left( \partial_1 h_1 \right)^2+ 2  \partial_1 h_1 \partial_2 h_1 \Cov(q_n (p),\hat{m}(X,n,r)) + \Var(\hat{m}(X,n,r))\, \left( \partial_2 h_1 \right)^2;  
\\ \Sigma^{(h,r)}_{22}&=\Var(\hat{m}(X,n,r))\, \left( \partial_2 h_2 \right)^2+ 2  \partial_1 h_2 \partial_2 h_2 \Cov(q_n (p),\hat{m}(X,n,r)) + \Var(q_n (p))\, \left( \partial_1 h_2 \right)^2; 
\\ \Sigma^{(h,r)}_{12} & = \Sigma^{(h,r)}_{21} 
\\ &= \Cov(q_n (p),\hat{m}(X,n,r)) \left( \partial_1 h_1 \partial_2 h_2 + \partial_2 h_1 \partial_1 h_2 \right)
+ \Var(q_n (p))\, \partial_1 h_1 \partial_1 h_2  + \Var(\hat{m}(X,n,r))\, \partial_2 h_1 \partial_2 h_2,  
\end{align*}
\vspace*{-1cm}
\begin{align*}
\quad  \text{with~} \quad \Var(q_n (p)) &= \frac{p(1-p)}{f_X^2(q_X(p))}, \quad   \Var(\hat{m}(X,n,r) = \Var(\lvert X - \mu \rvert^r +(2-r)(2F_X(\mu)-1)X), \text{~and~}
\\ \Cov(q_n (p),\hat{m}(X,n,r)) &=  \frac{\tau_r (\lvert X - \mu \rvert,p) +(2-r)(2F_X(\mu)-1) \tau_1 (p)}{f_X(q_X(p))}.
\end{align*}
The asymptotic correlation between the functional $h$ of the measure of dispersion and the sample quantile can be deduced from the above expressions.
In the specific case of having $\partial_2 h_1 = \partial_1 h_2 =0$, it is identical -up to its sign - whatever the choice of $h$ (under that restriction), namely
$$
\lim_{n \rightarrow \infty} \Cor\left(h\left(q_n(p), \hat{m}(X,n,r) \right)\right) =  \frac{\tau_r (\lvert X - \mu \rvert,p)  +(2-r)(2F(\mu)-1) \tau_1 (p)}{\sqrt{\Var(\lvert X - \mu \rvert^r) p(1-p)}}  \times \sgn(\partial_1 h_1 \partial_2 h_2).
$$
\end{corollary}
\end{appendices}
\end{document}